\documentclass[11pt,a4paper]{amsart}
\usepackage{fullpage}
\usepackage{ytableau}
\usepackage{tikz-cd}
\tikzcdset{scale cd/.style={every label/.append style={scale=#1},
		cells={nodes={scale=#1}}}}
\usepackage{tikz}
\usepackage{amssymb}
\usepackage{mathtools}

\newcommand{\al}{\alpha}

\newcommand{\m}{\mu}
\newcommand{\la}{\lambda}
\newcommand{\La}{\Lambda}

\DeclareMathOperator{\st}{SST_{\la}(\mu)}

\DeclareMathOperator{\im}{Im}
\DeclareMathOperator{\Hom}{Hom}
\DeclareMathOperator{\Ext}{Ext}
\DeclareMathOperator{\D}{\mathrm{\Delta}}

\newtheorem{theorem}{Theorem}[section]
\newtheorem{lemma}[theorem]{Lemma}
\newtheorem{proposition}[theorem]{Proposition}
\newtheorem{corollary}[theorem]{Corollary}

\newtheorem{definition}[theorem]{Definition}
\theoremstyle{remark}
\newtheorem{remark}[theorem]{Remark}

\newtheorem{example}[theorem]{Example}
\newtheorem{examples}[theorem]{Examples}
\newtheorem*{example*}{Example}
\numberwithin{equation}{section}
\begin{document}
	
		\begin{abstract}Consider the general linear group $G=GL_n(K)$ defined over an infinite field $K$ of positive characteristic $p$. We denote by $\Delta(\lambda)$ the Weyl module of $G$ which corresponds to a partition $\lambda$. Let $\lambda, \mu $ be partitions of $r$ and let $\gamma$ be a partition with parts divisible by $p$. In the first main result of this paper, we find sufficient conditions on $\lambda, \mu$ and $\gamma$ so that $\Hom_G(\Delta(\lambda),\Delta(\mu)) \simeq \Hom_G(\Delta(\lambda +\gamma),\Delta(\mu +\gamma))$, thus providing an answer to a question of D. Hemmer. As corollaries we obtain stability and periodicity results for homomorphism spaces. In the second main result we find related sufficient conditions on $\lambda, \mu$  and $p$ so that $\Hom_G(\Delta(\lambda),\Delta(\mu))$ is nonzero. An explicit map is provided that corresponds to the sum of all semistandard tableaux of shape $\mu$ and weight $\lambda$.
		\end{abstract}

	\title[On stability and nonvanishing of hom spaces]{On stability and nonvanishing of homomorphism spaces between Weyl modules}

	\author[Evangelou]{Charalambos Evangelou}
\address{Department of Mathematics, University of Athens, Greece}

\email{cevangel@math.uoa.gr}

\author[Maliakas]{Mihalis Maliakas}
\address{Department of Mathematics, University of Athens, Greece}

\email{mmaliak@math.uoa.gr}

\author[Stergiopoulou]{Dimitra-Dionysia Stergiopoulou}
\address{Department of Mathematics, University of Athens, Greece}

	\curraddr{School of Applied Mathematical and Physical Sciences, National Technical University of Athens, Greece}
\email{dstergiop@math.uoa.gr}



\date{September 1, 2024}
	\maketitle

\section{Introduction} Let $K$ be an infinite field of positive characteristic $p$ and let $G=GL_n(K)$ be the general linear group defined over $K$. We let $\Delta({\lambda})$ denote the Weyl module of $G$ corresponding to a partition $\lambda$. Since the classical papers of Carter-Luzstig \cite{CL} and Carter-Payne \cite{CP}, homomorphism spaces $\Hom_G(\Delta(\lambda), \Delta(\mu))$ between Weyl modules have attracted much attention. In those works sufficient arithmetic conditions on $\lambda, \mu$ and $p$ were found so that $\Hom_G(\Delta(\lambda), \Delta(\mu)) \neq 0$. The determination of the dimensions of these homomorphism spaces, or even when they are nonzero, seems a difficult problem. Other natural problems arise when one considers the behavior of $\Hom_G(\Delta(\lambda), \Delta(\mu))$ and higher extension groups under various operations in the representation theory of $G$, such as taking Frobenius twists of the modules \cite{CPSK}, or considering complements of the involved partitions \cite{FHKD}, or horizontal and vertical cuts \cite{FL}, \cite{D}, \cite{BG}.  

Various stability and periodicity properties in the modular representation theory of $G$ and the symmetric group that are related to adding powers of $p$ to the first parts of the involved partitions have  been studied. For example, decomposition numbers, $p$ - Kostka numbers, dimensions of various cohomology groups have been shown to exhibit such properties; see \cite{Gi}, \cite{Ha}, \cite{HK}, \cite{Hen}, \cite{NS}. In \cite[Theorem 1.1(1)]{MS4}, two of the present authors proved a periodicity property of the dimensions of $\Hom_G(\Delta(\lambda), \Delta(\mu))$ (and also higher extension groups) with respect to adding a power of $p$ to the first parts of $\lambda$ and $\mu$. In the first main theorem of the present paper (Theorem \ref{thm13})  we generalize this result by finding sufficient combinatorial hypotheses on $\lambda$ and $\mu$ so that $\dim\Hom_G(\Delta(\lambda), \Delta(\mu)) $ $=\dim\Hom_G(\Delta(\lambda +\gamma),\Delta(\mu + \gamma))$, where $\gamma$ is a partition such that each part is divisible by a  sufficiently high power of $p$. These powers are explicitly determined by $\lambda$ and $\mu$. Theorem \ref{thm13} provides an answer to a problem  of D. Hemmer \cite[Problem 5.4]{Hem2} formulated in terms of homomorphisms between Specht modules for the symmetric group. This is equivalent to the above formulation when $p>2$ via the Schur functor \cite{Gr}. 

If $k$ is a nonnegative integer and $\nu=(\nu_1,\dots,\nu_n)$ a partition, let $k\nu$ be the partition $(k\nu_1,\dots,k\nu_n).$ Let \[d_k=\dim \Hom _{G}(\Delta (\lambda + k \nu), \Delta(\mu + k \nu)).\]As an immediate corollary we have  that the sequence $d_{p}, d_{p^2}, \dots$ eventually stabilizes for any partition $\nu$ under certain conditions on $\lambda, \mu$ (Corollary \ref{stab}). Also, we show that for a particular class of partitions $\nu$, the sequence $d_{0}, d_1, \dots$ is periodic with period a power of $p$ (Corollary \ref{per}).

Our second main result (Theorem \ref{nonv1}) shows that under suitable combinatorial hypotheses on $\lambda, \mu$ and $p$ we have $\Hom_G(\Delta(\lambda), \Delta(\mu)) \neq 0$. A nonzero homomorphism $\Delta (\lambda) \to \Delta (\mu)$ is provided explicitly that corresponds to the sum of all semistandard tableaux of shape $\mu$ and weight $\lambda$. This generalizes Theorem 3.1 of \cite{MS3} that is valid for $\mu$ consisting of two parts. The second result is independent of the first, but both proofs use similar techniques. We rely on a systematic use of the classical presentation of the Weyl modules $\Delta(\lambda)$ and utilize combinatorics and computations involving tableaux.

The paper is organized as follows. In Section 2 we establish notation and gather preliminaries that will be needed later. In Section 3 we introduce a family of pairs of partitions $\lambda, \mu$ to which the two main theorems apply and we establish associated important combinatorial results that are utilized in the following sections. The proofs of the first and second main theorems are given in Section 4 and 5, respectively.

\section{Recollections}
The purpose of this section is to establish notation and recall facts that will be needed in the sequel.

 For the reader’s convenience we gather here some of the notation that will be used often. 
 \begin{center}
 	Dictionary of frequent notation
 \end{center}
{\small \begin{itemize}
	
 \item $\Lambda(n,r)$ and $\Lambda^+(n,r)$: set of sequences of length $n$ of nonnegative integers that sum to $r$ and the subset of partitions of $r$ (Subsection \ref{s2.1}).
	\item $D(\al)$, where $\alpha \in \Lambda(n,r)$: $D(\al)=D_{\al_1}V\otimes \dots \otimes D_{\al_n}V$ tensor product of divided powers of the natural $GL_n(K)$-module $V$ (Subsection \ref{s2.1}).
	\item $\pi_{\Delta(\mu)}:D(\mu) \to \Delta(\mu)$: natural projection onto the Weyl module $\Delta(\mu)$ (Subsection \ref{s2.1}).
	\item $\mathrm{SST}(\mu)$ and $\mathrm{SST}_{\alpha}(\mu)$: set of semistandard tableaux of shape $\mu$ and the subset of semistandard tableaux of shape $\mu$ and weight $\alpha$ (Subsection \ref{s2.2}).
	\item $\mathrm{RSST}(\mu)$ and $\mathrm{RSST}_{\alpha}(\mu)$: set of row semistandard tableaux of shape $\mu$ and the subset of row semistandard tableaux of shape $\mu$ and weight $\alpha$ (Subsection \ref{s2.2}).
	\item $A_T$, where $T \in \mathrm{RSST}_{\alpha}(\mu)$: the $n \times n$ matrix $A_T=(a_{ij})$, where $a_{ij}$ is equal to the number of appearances of the entry $j$ in the $i$th row of $T$ (Subsection 2.2).
	\item $[T]$, where $ T \in \mathrm{RSST}(\mu)$: $[T]=\pi_{\Delta(\mu)}(1^{(a_{11})} \cdots n^{(a_{1n})} \otimes \cdots \otimes 1^{(a_{n1})} \cdots n^{(a_{nn})})$, where $(a_{ij})$ is the matrix of $T$ (Subsection \ref{s2.2}). 
	\item $T[s,s+1]$: tableau consisting of rows $s$ and $s+1$ of $T$ (Subsection \ref{s2.2}).
	\item $\phi_T$, where $T \in \mathrm{RSST}_{\alpha}(\mu)$: the map $\phi_T: D(\alpha) \to \Delta(\mu)$ defined in Definition \ref{def2.4}.
	\item $\lambda({s,t})$, where $\lambda \in \Lambda^+(n,r)$: $\lambda({s,t})=(\lambda_1, \dots, \lambda_s+t, \lambda_{s+1}-t,\dots, \lambda_n)$ (Subsection \ref{s2.4}).
	\item $e^{\al}$, where $\alpha \in \Lambda(n,r)$: $e^{\al} =1^{(\al_1)} \otimes \cdots \otimes n^{(\al_n)} \in D(\al)$ if $\alpha =(\alpha_1, \dots, \alpha_n)$ (Subsection \ref{s2.4}).

	\item  $l_p(a)$: the least integer $i$ such that $p^i>a$ (Subsection \ref{s2.5}).
	\item $\psi_\alpha$: the map in Lemma \ref{lem12}(2).
	\item $\La^{+}(n,r)_g$: the subset of $\Lambda^+(n,r)$ defined in Definition \ref{def}.
	\item $c_s$ and $e_s$: the integers defined in Definition \ref{defce}.
	\item $\al^{\vee} =\al+\gamma$ (Subsection \ref{s4.2}).
	\item $\psi$: the map defined in Definition \ref{defpsi}.
	\item $\tau_s=t_1 + \dots +t_{s-1}$ (Subsection \ref{s5.2}).
\end{itemize}}
\subsection{Notation}\label{s2.1}
We will be working with homogeneous polynomial representations of $G=GL_n(K)$, where $K$ is an infinite field of characteristic $p>0$.  For a positive integer $r$, let $\Lambda(n,r)$ the set of sequences $\al=(\al_1, \dots, \al_n)$ of length $n$ of nonnegative integers that sum to $r$ and  $\Lambda^+(n,r)$ the subset of $\Lambda(n,r)$ consisting of partitions, that is sequences $\mu=(\mu_1, \dots, \mu_n)$ such that $\mu_1 \ge \mu_2 \ge \dots \ge \mu_n$. The length  of a partition $\mu$ is the number of nonzero parts of $\mu$ and is denoted by $\ell(\mu)$. If $m$ is a positive integer such that $m \le n$, we will consider $\La(m,r)$ and $\La^+(m,r)$ as subsets of $\La(n,r)$ and $\La^+(n,r)$, respectively by appending a string of $0$'s at the end if necessary.  

If $\al, \beta \in \Lambda(n,r)$, where $\al=(\al_1,\dots, \al_n), \ \beta = (\beta_1,\dots,\beta_n)$, we write $\al \unlhd \beta$ if $\al_1+\dots+\al_i \le \beta_1+\dots+\beta_i$ for all $i$.

We use the notation $\sum_i M_i$ for the direct sum of vector spaces $M_i$. We denote by $V=K^n$ be the natural $G$-module and $DV=\sum_{i\geq 0}D_iV$ the divided power algebra of $V$, (\cite[Section I.4]{ABW}). If $v \in V$ and $i$ is a nonnegative integer, we have the $i$th divided power $v^{(r)} \in D_iV$. We recall that if $i,j$ are nonnegative integers, then the product $v^{(i)}v^{(j)}$ of $v^{(i)}$ and $v^{(j)}$ is given by \[v^{(i)}v^{(j)}=\tbinom{i+j}{j}v^{(i+j)}\] where $\tbinom{i+j}{j}$ is the indicated binomial coefficient. If $\al=(\al_1,\dots, \al_n) \in \Lambda(n,r)$, let $D(\al)=D(\al_1,\dots,\al_n)$ be the tensor product $D_{\al_1}V\otimes \dots \otimes D_{\al_n}V$. If $\mu \in \Lambda^+(n,r)$, we denote by $\Delta(\mu)$ the corresponding Weyl module for $G$ and by \[{\pi}_{\D(\mu)} : D(\mu) \to \D(\mu)\] the natural projection.

\subsection{Semistandard tableaux}\label{s2.2}
	Consider the order $\epsilon_1 < \epsilon_2 <\dots< \epsilon_n$ on the natural basis $\{\epsilon_1,\dots,\epsilon_n\}$ of $V$. In the sequel we will denote each element $\epsilon_i$ by its subscript $i$. For a partition $\mu=(\mu_1,\dots,\mu_n) \in \Lambda^+(n,r)$, a tableau of shape $\mu$ is a filling of the diagram of $\mu$ with entries from $\{1,\dots,n\}$. A tableau is called \textit{row semistandard} if the entries are weakly increasing across the rows from left to right.  A row semistandard tableau is called \textit{semistandard} if the entries are strictly increasing in the columns from top to bottom. The set consisting of the semistandard (respectively, row semistandard) tableaux of shape $\mu$ will be denoted by $\mathrm{SST}(\mu)$ (respectively, $\mathrm{RSST}(\mu)$). The \textit{weight} of a tableau $T$ is the tuple $\alpha=(\alpha_1,\dots,\alpha_n)$, where $\alpha_i$ is the number of appearances of the entry $i$ in $T$. The set consisting of the semistandard (respectively, row semistandard) tableaux of shape $\mu$ and weight $\alpha$ will be denoted by $\mathrm{SST}_{\alpha}(\mu)$ (respectively, $\mathrm{RSST}_{\alpha}(\mu)$). For example, the following tableau of shape $\mu=(6,4)$
\[T= \begin{ytableau}
		\ 1&1&1&1&2&4\\
		\ 2&2&4&4 
	\end{ytableau} \]
is semistandard and has weight $\alpha=(4,3,0,3)$. We will use `exponential' notation for row semistandard tableaux. For the previous example we may write \[ T=\begin{matrix*}[l]
	1^{(4)} 2 4 \\
	2^{(2)}4^{(2)}  \end{matrix*}. \] 
\begin{remark}\label{obv}For latter use we make the following remark. If a tableau $T$ of shape $(\mu_1,\dots,\m_n)$ and weight of the form $(\la_1,\dots,\la_{s}+t, \la_{s+1}-t, \dots,\la_n)$, has the property that all elements $1,2,\dots,s$ appear in the first $s$ rows of $T$, then \[t \le (\mu_1 - \lambda_1)+\dots+(\mu_s - \lambda_s).\] Indeed, let $k$ be the number of entries of $T$ that are less than or equal to $s$. Since these entries appear in the first $s$ rows of $T$ we have $k \le \mu_1+\cdots +\mu_s.$ On the other hand, since the weight of $T$ is $(\la_1,\dots,\la_{s}+t, \la_{s+1}-t, \dots,\la_n)$, we have $k=\lambda_1+ \cdots +\lambda_s + t.$ Thus $t \le (\mu_1 - \lambda_1)+\dots+(\mu_s - \lambda_s)$.
\end{remark}
For $T \in \mathrm{RSST}_{\alpha}(\mu)$ we may associate an $n \times n$ matrix $A_T=(a_{ij})$ with nonnegative integer entries by letting $a_{ij}$ be the number of appearances of the entry $j$ in the $i$th row of $T$. It is understood that $a_{ij}=0$ if $i>\ell(\mu)$. Let  $M_{n}(\mathbb{N})$ be the set of $n \times n$ matrices with nonnegative integer entries.  For $A=(a_{ij}) \in M_{n}(\mathbb{N})$, we have the sequences $A^{(1)}, A^{(2)} \in \La(n)$ of column sums and row sums of $A$ defined by \begin{center}
	$A^{(1)}=(\sum_ia_{i1},\dots,\sum_ia_{in})$, \  $A^{(2)}=(\sum_ja_{1j},\dots,\sum_ja_{nj})$.
\end{center} It is clear from the definitions that we have a bijection \begin{equation}\label{bijection1}\mathrm{RSST}_{\alpha}(\mu) \to \{A \in M_{n}(\mathbb{N}): A^{(1)} =\al, A^{(2)} =\mu \},\  T \mapsto A_T.\end{equation} For $B \in  M_{n}(\mathbb{N})$ such that $B^{(1)} =\al, B^{(2)} =\mu$, we denote by $T_B \in \mathrm{RSST}_{\alpha}(\mu)$ the unique row semistandard tableau such that $A_{T_B}=B$. 

If $T \in \mathrm{RSST}_{\alpha}(\mu)$ has corresponding matrix $A_T=(a_{ij})$, we may consider the element \[1^{(a_{11})} \cdots n^{(a_{1n})} \otimes \cdots \otimes 1^{(a_{n1})} \cdots n^{(a_{nn})} \in D(\mu).\] The image of this element under the natural projection ${\pi}_{\D(\mu)} : D(\mu) \to \D(\mu)$ will be denoted by $[T]$, \[[T]=\pi_{\Delta(\mu)}(1^{(a_{11})} \cdots n^{(a_{1n})} \otimes \cdots \otimes 1^{(a_{n1})} \cdots n^{(a_{nn})}).\] We recall from \cite{ABW} the following result.
\begin{theorem}\label{standard}
	A basis of the vector space $\Delta(\mu)$, where $\mu \in \La^{+}(n,r)$, is the set $\{[T]: T \in \mathrm{SST}(\mu)\}.$
\end{theorem}
Let $T \in \mathrm{RSST}(\mu)$ and consider the tableau \[T[s,s+1]\] consisting of rows $s$ and $s+1$ of $T$, where $s \in \{1,2,\dots,m-1\}$. We have the partition $\mu[s,s+1]=(\mu_s, \mu_{s+1})$ and the corresponding Weyl module $\D(\mu[s,s+1])$. From the analog of \cite[Lemma II.2.3]{ABW} for Weyl modules we obtain the following result.
\begin{lemma}\label{insertrows}
Let $T \in \mathrm{RSST}(\mu)$. If in $\D(\mu[s,s+1])$ we have $[T[s,s+1]]=\sum_ic_i[T[s,s+1]_i],$ where $c_i \in K$, then in $\D(\mu)$ we have $[T]=\sum_ic_i[T_i],$ where $T_i$ is the tableau obtained from $T$ by replacing rows $s$ and $s+1$ with $T[s,s+1]_i$.
\end{lemma}
\subsection{The maps $\phi_T$ and weight subspaces of $\Delta(\mu)$}\label{s2.3} Let $\al \in \La(n,r)$, $\mu \in \La^{+}(n,r)$ and $T \in  \mathrm{RSST}_{\al}(\mu)$ with corresponding matrix $A=A_T=(a_{ij})$. We have $A^{(1)}=\al $ and $A^{(2)}=\mu$. For each $j=1,2,\dots,n$, consider the indicated component \begin{equation}\label{D} \Delta: D(\al_j) \to D(a_{1j},a_{2j},\dots,a_{nj}),\end{equation} of the comultiplication map of the Hopf algebra $DV$. If $x \in D(\al_j)$, the image $\D(x) \in D(a_{1j},a_{2j},\dots,a_{nj})$ is a sum of elements of the form $x_s(a_{1j},1) \otimes x_s(a_{2j},2) \otimes \cdots \otimes x_s(a_{nj},n)$, where for each $j$ we have $x_s(a_{ij},i) \in D(a_{ij})$. By a slight abuse of notation we will write $x_s(a_{ij})$ in place of $x_s(a_{ij},i)$. Thus we will write $\Delta(x) = \sum_{s}x_s(a_{1j}) \otimes x_s(a_{2j}) \otimes \cdots \otimes x_s(a_{nj})$.
\begin{definition}\label{def2.4}
	Let $T \in \mathrm{RSST}_{\al}(\mu)$. With the previous notation, define the map $ \phi_{T}:D(\al) \to \Delta(\mu)$ that sends  $x_1 \otimes x_2 \otimes \dots\otimes x_n$ to \[\pi_{\D(\mu)} \big( \sum_{s_1,\dots,s_n}x_{1s_1}(\al_{11}) \cdots x_{ns_n}(\al_{1n}) \otimes  \cdots \otimes x_{1s_1}(\al_{n1})\cdots x_{ns_n}(\al_{nn}) \big). \]
\end{definition}
\begin{remark}\label{remarkphi}
	We note that $\phi_T(1^{(\al_1)} \otimes \cdots \otimes n^{(\al_n)} )=[T]$ if $T \in \mathrm{RSST}_{\al}(\mu)$, where $\al = (\al_1,\dots,\al_n).$
\end{remark}
\begin{example*}
	Let $n=3$ and consider the partition $\mu=(6,5)$ and the row semistandard tableau \[ T=\begin{matrix*}[l]
		1^{(2)} 2^{(4)} \\
		12^{(2)}3^{(2)}  \end{matrix*} \in \mathrm{RSST}_{\al}(\mu), \] 
that has weight $\alpha=(3,6,2)$. Then the matrix of $T$ is \[A_T=\begin{pmatrix*}
	2& 4& 0 \\ 1& 2& 2 \\ 0& 0& 0
 \end{pmatrix*}.\]Let $x=x_1 \otimes x_2 \otimes x_3 \in D(\alpha)$, where $x_1=2^{(3)}, \; x_2= 1^{(3)}2^{(3)}, \;  x_3=3^{(2)}.$
Computing the images of the $x_j$ under the maps (\ref{D}) we have \begin{align*}
	&D(3) \to D(2,1), 2^{(3)} \mapsto 2^{(2)} \otimes 2,\\&
	D(6) \to D(4,2), 1^{(3)}2^{(3)} \mapsto 1^{(3)}2 \otimes 2^{(2)} +1^{(2)}2^{(2)} \otimes 12 + 12^{(3)} \otimes 1^{(2)},\\&
	D(2) \to D(2), 3^{(2)} \mapsto 3^{(2)}.
\end{align*} Hence for the image of $x$ under the map $\phi_T: D(\alpha) \to \Delta(\mu)$ of Definition \ref{def2.4} we have \begin{align*}
\phi_T(x)&={\pi}_{\D(\mu)}(2^{(2)}1^{(3)}2 \otimes 2 2^{(2)}  3^{(2)}) + {\pi}_{\D(\mu)}(2^{(2)} 1^{(2)}2^{(2)}\otimes 2 12 3^{(2)}) \\& \; \; \;  \; +{\pi}_{\D(\mu)}(2^{(2)} 12^{(3)}\otimes 2 1^{(2)} 3^{(2)}) \\&= \tbinom{2+1}{1} \tbinom{1+2}{2}	\begin{bmatrix*}[l]
	1^{(3)}2^{(3)}  \\
	2^{(3)} 3^{(2)}
\end{bmatrix*} +\tbinom{2+2}{2} \tbinom{1+1}{1}	\begin{bmatrix*}[l]
1^{(2)}2^{(4)}  \\
12^{(2)} 3^{(2)}
\end{bmatrix*} + \tbinom{2+3}{3}	\begin{bmatrix*}[l]
12^{(2)}3^{(3)}  \\
1^{(2)}2 3^{(2)}
\end{bmatrix*}.
\end{align*}
The binomial coefficients come from the multiplication in the divided power algebra $DV$.	
\end{example*}
Let $\D(\mu)_{\al}$ be the weight subspace of $\Delta(\mu)$ that corresponds to the weight $\al$. We recall from \cite[Section 2]{AB} the following.
\begin{proposition}[\cite{AB}] Let $\al \in \Lambda(n,r)$ and $\mu \in \Lambda^+(n,r)$. Then there is an isomorphism of vector spaces $\D(\mu)_{\al} \simeq \Hom_G(D(\al), \Delta(\mu))$ such that $[T] \mapsto \phi_T$ for all $T \in \mathrm{RSST}_{\al}(\mu)$. Moreover, a basis of $\Hom_G(D(\al), \Delta(\mu))$ is the set $\{\phi_T: T \in \mathrm{SST}_{\al}(\mu)\}.$
\end{proposition}

\subsection{Presentation of $\Delta(\la)$}\label{s2.4}
First some notation. If $\al=(\al_1,\dots,\al_n) \in \La(n,r)$ and $s, t$ are integers such that $1 \le s \le n-1$ and $1 \le t \le \al_{s+1} $, let us denote the sequence $(\al_1,\dots,\al_s+t, \al_{s+1}-t,\dots,\al_n) \in \La(n,r)$ by $\al({s,t})$. 

Recall from \cite[Theorem II.3.16]{ABW} that we have the following presentation of $\Delta(\la)$, \begin{equation}\label{pres1}
\sum_{s=1}^{n-1}\sum_{t=1}^{\lambda_{s+1}}D(\la(s,t)) \xrightarrow{\square_{\la}} D(\lambda) \xrightarrow{\pi_{\D(\la)}} \Delta(\lambda) \to 0,\end{equation}
where the restriction of $ \square_{\la} $ to the summand $D(\la(s,t))$ is the composition
\begin{equation}\label{boxmap}
\square_{\la, s,t}:D(\la(s,t)) \xrightarrow{1\otimes\cdots \otimes \Delta \otimes \cdots 1}D(\lambda_1,\dots,\lambda_s,t,\lambda_{s+1}-t,\dots,\lambda_m) \xrightarrow{1\otimes\cdots \otimes \eta \otimes \cdots 1} D(\lambda),
\end{equation}
where $\Delta:D(\lambda_s+t) \to D(\lambda_s,t)$ and $\eta:D(t,\lambda_{s+1}-t) \to D(\lambda_{s+1})$ are the indicated components of the comultiplication and multiplication respectively of the Hopf algebra $DV$. 

If $\al=(\al_1,\dots,\al_n) \in \La(n,r)$, we let \[e^{\al}=1^{(\al_1)} \otimes \cdots \otimes n^{(\al_n)} \in D(\al).\] Let $T \in \mathrm{RSST}_{\la}(\mu)$. Then \[ T=\begin{matrix*}[l]
	1^{(a_{11})}2^{(a_{12})} \cdots n^{(a_{1n})} \\
	
	\hspace{40pt}\cdots\\
	1^{(a_{n1})} 2^{(a_{n2})}\cdots n^{(a_{nn})}  \end{matrix*}, \]
where $A_T=(a_{ij})$. In many occasions we will have subscripts of the form $a_{ij}$. For clarity we will often write $a_{i,j}$ in place of $a_{ij}$ when $i$ and $j$ are more complicated expressions. 

The next lemma concerns the composition of the maps $\square_{\la, s, t}$ and $\phi_T$.
\begin{lemma}\label{comp2}
	Let  $T \in \mathrm{RSST}_{\la}(\mu)$, $A_T=(a_{ij})$, $s \in \{1,2,\dots,n-1\}$ and $t \in \{1,2,\dots,\la_{s+1}\}$. With the previous notation, the image of $e^{\la(s,t)}$ under the composition $D(\la(s,t)) \xrightarrow{\square_{\la, s,t}} D(\la) \xrightarrow{\phi_T} \Delta(\mu) $ is equal to \begin{equation}\label{eq11}
		\sum_{t_1,\dots,t_{n}}\tbinom{a_{1s}+t_1}{t_1}\tbinom{a_{2s}+t_2}{t_2} \cdots \tbinom{a_{ns}+t_n}{t_n}[T(s,t;t_1,\dots,t_{n})],\end{equation}
	where the matrix  of the row semistandard tableau $T(s,t;t_1,\dots,t_{n})$ is obtained from $(a_{ij})$ by replacing the columns $s$ and $s+1$ by the columns \[\begin{pmatrix}a_{1s}+t_1 \\ \vdots \\ a_{ns}+t_n \end{pmatrix}, \begin{pmatrix}a_{1,s+1}-t_1 \\ \vdots \\ a_{n,s+1}-t_n \end{pmatrix}\]
	and the sum is over all $t_1,\dots,t_{n} \in \mathbb{N}$ such that \begin{align}\label{ineq111}&t_{1}+\dots+t_{n}=t, \; a_{1,s+1}-t_1 \ge 0, \; \dots, \; a_{n,s+1}-t_{n} \ge 0.\end{align}
\end{lemma}
\begin{proof}
From the definition of $\square_{\la, s,t}$ given in (\ref{boxmap}), we have \begin{align*}\square_{\la, s,t}(e^{\la(s,t)})=&1^{(\lambda_1)}\otimes \cdots \otimes s^{(\lambda_s)}\otimes s^{(t)}(s+1)^{(\lambda_{s+1}-t)}\\&\otimes (s+2)^{(\lambda_{s+2})} \otimes \cdots \otimes n^{(\lambda_n)}.\end{align*}
Let $x_{s+1} = s^{(t)}(s+1)^{(\lambda_{s+1}-t)}$. For the map \[\Delta:  D(\lambda_{s+1}) \to D(a_{1,s+1}, a_{2,s+1} \dots, a_{n, s+1})\] given (\ref{D}) we have \[\Delta(x_{s+1})=\sum_{t_1,\dots,t_n}s^{(t_1)}(s+1)^{(a_{1,s+1}-t_1)}\otimes \cdots \otimes s^{(t_n)}(s+1)^{(a_{n,s+1})-t_n},\] where the sum is as in (\ref{ineq111}). Now Definition \ref{def2.4} yields  \begin{align*}
	\phi_T \circ \square_{\la, s,t}(e^{\la(s,t)})&=\sum_{t_1,\dots, t_n} \begin{bmatrix*}[l]
		1^{(a_{11})}\cdots s^{(a_{1s})}s^{(t_1)}(s+1)^{(a_{1,s+1}-t_{s+1})} \cdots n^{(a_{1n})} \\
		\hspace{40pt}\cdots\\
		1^{(a_{n1})} \cdots s^{(a_{ns})}s^{(t_n)}(s+1)^{(a_{n,s+1}-t_{s+1})}\cdots n^{(a_{nn})}  \end{bmatrix*} \\&=\sum_{t_1,\dots,t_{n}}\tbinom{a_{1s}+t_1}{t_1} \cdots \tbinom{a_{ns}+t_n}{t_n}\\&\hspace{40pt} 
		\begin{bmatrix*}[l]
			1^{(a_{11})}\cdots s^{(a_{1s}+t_1)}(s+1)^{(a_{1,s+1}-t_{s+1})} \cdots n^{(a_{1n})} \\
			\hspace{40pt}\cdots\\
			1^{(a_{n1})} \cdots s^{(a_{ns}+t_n)}(s+1)^{(a_{n,s+1}-t_{s+1})}\cdots n^{(a_{nn})}  \end{bmatrix*}\\&=\sum_{t_1,\dots,t_{n}}\tbinom{a_{1s}+t_1}{t_1} \cdots \tbinom{a_{ns}+t_n}{t_n}[T(s,t;t_1,\dots,t_{n})].
\end{align*} \end{proof}
We note that $T(s,t;t_1,\dots,t_{n})$ has weight $\la(s,t)$.
\begin{corollary}\label{corhom} Let $T \in \mathrm{RSST}_{\la}(\mu)$, $A_T=(a_{ij})$, $s \in \{1,2,\dots,n-1\}$ and $t \in \{1,2,\dots,\la_{s+1}\}$. The image of $\phi_T$ under map 
\[\Hom_G(D(\lambda),\Delta(\mu)) \xrightarrow{\Hom_G(\square_{\la, s, t},\Delta({\mu}))} \Hom_G(D(\la({s,t})),\Delta(\mu)),\] 
is equal to $\sum_{t_1,\dots,t_{n}}\tbinom{a_{1s}+t_1}{t_1}\tbinom{a_{2s}+t_2}{t_2} \cdots \tbinom{a_{ns}+t_n}{t_n}\phi_{T(s,t;t_1,\dots,t_{n})}$,
where the sum is over all $t_1,\dots,t_{n} \in \mathbb{N}$ satisfying conditions (\ref{ineq111}) of the previous lemma.
\begin{proof}
	Since $D(\la(s,t))$ is a cyclic $G$-module generated by the element $e^{\la(s,t)}$,  see \cite[Theorem A4]{SY}, it suffices to show that $\Hom_G(\square_{\la, s, t},\Delta({\mu}))(\phi_T)(e^{\la(s,t)})$$=f(e^{\la(s,t)}),$ where $f=\sum_{t_1,\dots,t_{n}}\tbinom{a_{1s}+t_1}{t_1}\tbinom{a_{2s}+t_2}{t_2} $ $\cdots \tbinom{a_{ns}+t_n}{t_n}\phi_{T(s,t;t_1,\dots,t_{n})}$. This equality follows from the previous lemma and Remark \ref{remarkphi}.
\end{proof}
\end{corollary}
\begin{remark}\label{remnewcond}	We will need a slightly different expression for the sum (\ref{eq11}) in Lemma \ref{comp2} in the following special case. With the notation and assumptions of Lemma \ref{comp2}, suppose in addition that $A_T=(a_{ij})$ is upper triangular and that $s$ satisfies $s<m$, where $m =\ell(\mu)$. Define \begin{center}
		$\tau_0=0$ and $\tau_i=t_1+\dots+t_i$ for all $i=1,\dots,s.$
	\end{center} Then (\ref{eq11}) is equal to  
	\begin{equation}\label{newsum}
		\sum_{t_1,\dots,t_{s}}\tbinom{a_{1s}+t_1}{t_1}\tbinom{a_{2s}+t_2}{t_2} \cdots \tbinom{a_{ss}+t_s}{t_s}[T(s,t;t_1,\dots,t_{s},t-\tau_s,0\dots,0],\end{equation} where the sum is over all $t_1,\dots,t_s \in \mathbb{N}$ such that 
	\begin{equation}\label{newcond}
		t-\tau_{i-1}-\sum_{u=i+1}^{s+1}a_{u,s+1} \le t_i \le \min\{a_{i,s+1}, t-\tau_{i-1}\}, \; i=1,\dots,s.
	\end{equation}
	Indeed, since $(a_{ij})$ is upper triangular, the inequalities in (\ref{ineq111}) yield $t_{s+2}=\dots=t_n=0$ and thus $t_{s+1}=t-\tau_{s}$. It is clear that (\ref{ineq111}) imply the right inequalities of (\ref{newcond}). Also from (\ref{ineq111}) we have $\sum_{u=i+1}^{s+1}a_{u,s+1} \ge \sum_{u=i+1}^{s+1}t_{u}$ $=t-\tau_i$ and hence the left inequalities of (\ref{newcond}) hold. Conversely, suppose (\ref{newcond}) hold. By defining $t_{s+1}=t-\tau_s$ it follows from the right inequality (for $i=s$) that $t_{s+1} \ge 0$. It is clear that the right inequalities (\ref{newcond}) imply the inequalities in (\ref{ineq111}).
	\end{remark} 

We will need the following from \cite[Lemma 4.2]{MS3}.
\begin{lemma}\label{lem15}Let $\nu \in \La^+(2,c)$, $\nu=(\nu_1,\nu_2)$ and $T=\begin{matrix*}[l]
		1^{(a_1)}2^{(a_2)}  \cdots   n^{(a_n)} \\
		1^{(b_1)}2^{(b_2)}  \cdots  n^{(b_n)}
	\end{matrix*} \in \mathrm{Tab} (\nu)$.
	Then we have the following identities in $\Delta{(\nu)}$.
	\begin{enumerate}
		\item If $a_1+b_1>\nu_1$, then $[T]=0$.
		\item If $a_1+b_1 \le \nu_1$, then 
		\[[T]=(-1)^{b_1}\sum_{k_2,\dots,k_n}\tbinom{b_2+k_2}{b_2}\cdots\tbinom{b_n+k_n}{b_n}
		\begin{bmatrix*}[l]
			1^{(a_1+a_2)}2^{(a_2-k_2)}  \cdots   n^{(a_n-k_n)} \\
			2^{(b_2+k_2)}  \cdots  n^{(b_n+k_n)}
		\end{bmatrix*},
		\] where the sum ranges over all  $k_2,\dots,k_n \in \mathbb{N}$
		such that  $k_2+\dots+k_n=b_1 $ and $k_s \le a_s$ for all $s=2,\dots,n$.	\end{enumerate}	
\end{lemma}

\subsection{Binomial coefficients}\label{s2.5}
Our convention for binomial coefficients is that $\tbinom{a}{b}=0$ if $b>a$ or $b<0$.

If $a$ is a positive integer, let $l_p(a)$ be the least integer $i$ such that $p^i>a$. We will need the following well known divisibility properties of binomial coefficients. See, for example, Lemma 22.4 and Corollary 22.5 of \cite{Jam}.
\begin{lemma}\label{bin1}
	Let $a \ge b \ge 1$ be integers. \begin{enumerate}
		\item If $d$ is a positive integer such that $ p^{l_p(b)}$ divides $d$, then $\tbinom{a+d}{b} \equiv \tbinom{a}{b} \mod p $.
		\item $p$ divides all of $\tbinom{a+1}{1}, \tbinom{a+2}{2}, \dots, \tbinom{a+b}{b}$ if and only if $p^{l_{p}(b)}$ divides $a+1$.
	\end{enumerate}
\end{lemma}

We will also need the following identities that were used in \cite{MS3}. 
\begin{lemma}\label{bin2} \;\begin{enumerate}
		
		\item (Vandermonde) Let $a_1, \dots, a_s \in \mathbb{N}$ and $a=a_1+\cdots+a_s$. Then \\$\sum_{t_1,\cdots,t_s}$ $\tbinom{a_1}{ t_1}\cdots \tbinom{a_s}{t_s}=\tbinom{a}{t},$
			where the sum ranges over all $t_1,\dots,t_s \in \mathbb{N} $ such that $t_1+\cdots+t_s = t.$

		\item Let $a,b,c \in \mathbb{N}$  such that $ b \le a$. Then
		$\sum_{j=0}^{c}(-1)^{c-j}\tbinom{a+j}{j}\tbinom{b}{c-j}$ $=\tbinom{a-b+c}{c}$ $= \sum_{j=0}^{c}(-1)^{j}\tbinom{a+c-j}{c-j}\tbinom{b}{j}.$
	\end{enumerate}
\end{lemma}

\section{The set $P(\mu)$ and combinatorics}
	
	The key definition of this paper is the following. \begin{definition}\label{defP}Fix $\mu\in \La^+(n,r)$ and let $m=\ell(\mu)$. If $m=1$, let $P(\mu) =\La(n,r)$. If $m \ge 2$, let $P(\mu)$ consist of all $\al \in \La(n,r)$ such that \begin{equation}\label{in3.1}\mu_1 +\mu_2 + \dots +\widehat{\mu}_{i-1}+\mu_{i} \le \al_1+\al_2+\dots+\al_{i-1},\end{equation}
	for all $i=2,\dots,m$, where $\widehat{\mu}_{i-1}$ means that ${\mu}_{i-1}$ is omitted. \end{definition}
\begin{example*}Suppose $\mu \in \Lambda^+(n,r)$ has length $m=4$ and let $ \alpha \in \Lambda(n,r)$.  According to the above definition we have that 
	$ \alpha \in P(\mu)$ if and only if 
\begin{align*} \mu_2 &\le \alpha_1, \\ \mu_1+\mu_3 &\le \alpha_1 +\alpha_2, \\ \mu_1 + \mu_2 +\mu_4 &\le \alpha_1 +\alpha_2 + \alpha_3. \end{align*}	
Let $T$ be semistandard tableau of shape $\mu$ and weight $\alpha$, where $\alpha \in P(\mu)$. 
Then $T$ looks like 
\begin{center}
	\begin{tikzpicture}[scale=0.7]
		\draw (1,0) -- (12,0) -- (12,1) -- (1,1) -- (1,0);
		\draw (1,-1) -- (9,-1) -- (9,0) -- (1,0) -- (1,-1);
		\draw (1,-2) -- (6,-2) -- (6,-1) -- (1,-1) -- (1,-2);
		\draw (1,-3) -- (4,-3) -- (4,-2) -- (1,-2) -- (1,-3);
		\draw[] (9.8,0) --(9.8,1);
		\draw[] (7.3,-1) --(7.3,0);
		\draw[] (4.5,-2) --(4.5,-1);
		\node (A) at (13,0.5) {$\mu_1$};
		\node (B) at (13,-0.5) {$\mu_2$};
		\node (C) at (13,-1.5) {$\mu_3$};
		\node (C) at (13,-2.5) {$\mu_4$};
		\node (K) at (5.5,0.5) {$1$};
		\node (X) at (10,0.5) {};
		\node (Y) at (0.8,0.5) {};
			\node (Z) at (0.8,-0.5) {};
		\draw [-,dashed] (K) -- (X);
		\draw [-,dashed] (K) -- (Y);
		\node (K2) at (4.2,-0.5) {$2$};
		\node (X2) at (7.5,-0.5) {};
		\node (Y2) at (0.8,-0.5) {};
			\node (K32) at (2.7,-1.5) {$3$};
		\node (X32) at (4.6,-1.5) {};
		\node (Y32) at (0.8,-1.5) {};
		\draw [-,dashed] (K2) -- (X2);
		\draw [-,dashed] (K2) -- (Y2);
		
			\draw [-,dashed] (K32) -- (X32);
		\draw [-,dashed] (K32) -- (Y32);
		\node (Adot) at (11,0.5) {$\dots$};
		\node (Bdot) at (8.2,-0.5) {$\dots$};
		\node (Cdot) at (5.3,-1.5) {$\dots$};
		\node (Ddot) at (2.6,-2.5) {$\dots$};
	\end{tikzpicture}
\end{center}
meaning that the number of $i$'s in the $i$th row is at least $\mu_{i+1}$, where $i=1,2,3$. This example is generalized in Lemma \ref{lem12}(1) below. \end{example*}

We have the following properties of $P(\mu)$.	
	\begin{lemma}\label{rem11}Let $\mu\in \La^+(n,r)$ and $m=\ell(\mu) \ge 2$. If $\al \in P(\mu)$, then the following hold. \begin{enumerate} \item $\alpha_{i} \ge (\mu_1+\dots+\mu_{i-1})-(\alpha_1 +\dots+\alpha_{i-1})$ for all $i=2,\dots,m-1$. \item If $\al \unlhd \mu$, then $\alpha_i \ge \mu_{i+1}$ for all $i=1,2,\dots,m-1$.  \item If $\beta \in \La(n,r)$ and $\al \unlhd \beta$, then $\beta \in P(\mu)$. \item If $\gamma \in \Lambda^+(n)$ and $\ell(\gamma) \le m$, then $\alpha+\gamma \in P(\mu +\gamma)$.\end{enumerate} \end{lemma}
	\begin{proof}
		(1) Let $ i \in \{2, \dots, m-1\}$. Using inequality (\ref{in3.1}) for $i+1$ in place of $i$ we have \begin{equation}\label{in3.2}\mu_1 +\mu_2 + \dots +\widehat{\mu}_{i}+\mu_{i+i} \le \al_1+\al_2+\dots+\al_{i}\end{equation} and therefore
		\[\mu_1 +\mu_2 + \dots +\mu_{i-1} \le \al_1+\al_2+\dots+\al_{i}.\] Hence  
	$\alpha_{i} \ge (\mu_1+\dots+\mu_{i-1})-(\alpha_1 +\dots+\alpha_{i-1})$.
	
	(2)	We have $\alpha_1 \ge \mu_2$ from inequality (\ref{in3.1}). Let $ i \in \{2, \dots, m-1\}$. Using inequality (\ref{in3.2}) and the assumption $\al \unlhd \mu$ we have \begin{align*}
		\alpha_i &\ge \mu_1 +\mu_2 + \dots +\widehat{\mu}_{i}+\mu_{i+i} -(\alpha_1 +\dots+\alpha_{i-1})\\ &\ge \mu_1 +\mu_2 + \dots +\widehat{\mu}_{i}+\mu_{i+i} -(\mu_1 +\dots+\mu_{i-1})\\&=\mu_{i+1}.
	\end{align*}
	
	(3) Let $ i \in \{2, \dots, m\}$. From inequality (\ref{in3.1}) and the hypothesis $\al \unlhd \beta$ we have \[ \mu_1 +\mu_2 + \dots +\widehat{\mu}_{i-1}+\mu_{i} \le \al_1+\al_2+\dots+\al_{i-1} \le \beta_1+\beta_2+\dots+\beta_{i-1}.\] Hence $\beta \in P(\mu)$.
	
	(4) Let $ i \in \{2, \dots, m\}$. Since $\gamma$ is a partition we have $\gamma_i \le \gamma_{i-1}$ and thus \[\gamma_1+\cdots +\gamma_{i-2} + \gamma_i \le \gamma_1+ \cdots +\gamma_{i-2} +\gamma_{i-1}. \] From this and inequality (\ref{in3.1}) we obtain
	 \begin{align*}
		&(\mu_1+\gamma_1)+\cdots +(\mu_{i-2}+\gamma_{i-2})+ (\widehat{\mu}_{i-1}+\widehat{\gamma}_{i-1})+(\mu_{i} + \gamma_i) \\ \le& (\alpha_1 + \gamma_1) +\dots+ (\alpha_{i-2}+\gamma_{i-2}) + (\al_{i-1}+\gamma_{i-1}).
	\end{align*}
	Hence $\alpha + \gamma \in P(\mu +\gamma)$.\end{proof}
If $T$ is a semistandard tableau of shape $\mu$ and weight $\alpha$, then since the entries of $T$ are strictly increasing in each column from top to bottom, it is clear that each entry $i$ of $T$ appears only in then first $i$ rows of $T$. If in addition $\alpha \in P(\mu)$, then we have a result in the converse direction according to part (1) of the next lemma; if $T$ is row semistandard and each element $i$ of $T$ occurs only in the first $i$ rows of $T$, then $T$ is semistandard.  As a consequence we obtain the isomorphism of weight spaces of part (2) of the next lemma which is crucial for our purposes.
	\begin{lemma}\label{lem12}Let $\mu\in \La^+(n,r)$ and $m=\ell(\mu) \ge 2$. Let $\al \in P(\mu)$ and $\gamma \in \La^{+}(n)$ such that $\ell(\gamma) \le m$.
	\begin{enumerate}	
	\item  Suppose $S \in \mathrm{Tab}_{\al+\gamma}(\mu+\gamma)$ has the property that for each $i=1,\dots,m-1$, all entries of $S$ equal to $i$ are contained in the first $i$ rows of $S$. Then for each $i=1,\dots,m-1$, the $i$th row of $S$ contains the element $i$ at least $\gamma_i+\mu_{i+1}$ times.
		
			If, in addition, $S$ is row semistandard, then $S$ is semistandard.
			
			\item Suppose $\ell(\gamma) < m$. Then the map $\mathrm{SST}_{\al}(\mu) \to \mathrm{SST}_{\al+\gamma}(\mu+\gamma), \ T \mapsto T^{\vee}$, is a bijection, where $T^{\vee}$ is obtained from $T$ by inserting $\gamma_i$ copies of $i$ at the beginning of row $i$, for each $i=1,2,\dots,m-1$. Hence we have an isomorphism of vector spaces \[\psi_{\al}: \Hom_G(D(\al),\Delta(\mu)) \to \Hom_{G}(D(\al+\gamma),\Delta(\mu+\gamma))\] such that $\psi_{\al}(\phi_T)=\phi_{T^{\vee}},$ where $T \in \mathrm{SST}_{\al}(\mu).$ \end{enumerate}	\end{lemma}
		\begin{proof}(1) Let $i \in \{1,2,\dots,m-1\}$. The first conclusion of (1) is clear if $i=1$, since from the assumption we have that all elements of $S$ that are equal to $1$ are located in the first row and the number of these is equal to $\al_1+\gamma_1 $ and we have  $\al_1 \ge \mu_2$. So let $i>1$. Suppose that the number $k$ of elements in row $i$ of  $S$ that are equal to $i$ satisfies $k<\gamma_i +\mu_{i+1}$. From the assumption we have that in rows $1,2,\dots,i-1$ of $S$, the number of elements equal to $i$ is $\al_i+\gamma_i-k$. Hence the total number of elements in rows $1,2,\dots, i-1$ of $S$ is greater than or equal to \begin{align*}
				&(\al_1 +\gamma_1) + \cdots + (\al_{i-1}+\gamma_{i-1})+(\al_i+\gamma_i-k) \\&=(\al_1+\cdots+\al_{i-1}+\al_i) +(\gamma_1+\cdots+\gamma_{i-1})+\gamma_i-k \\ &>(\al_1+\cdots+\al_{i-1}+\al_i) +(\gamma_1+\cdots+\gamma_{i-1})-\mu_{i+1} \\ &\ge (\mu_1+\cdots+\mu_{i-1}+\mu_{i+1})+(\gamma_1+\cdots+\gamma_{i-1})-\mu_{i+1}\\&=(\mu_1+\gamma_1)+\cdots+(\mu_{i-1}+\gamma_{i-1}).\end{align*} This is not possible because of the strict inequality.
			
			Suppose in addition that $S$ is row semistandard. From the first statement of (1) we have that the $i$th row of $S$ contains at least $\gamma_i + \mu_{i+1}$ copies of $i$, for each $i=1,\dots,m-1$. Since $\gamma_i + \mu_{i+1} \ge \gamma_{i+1}+\mu_{i+1}$, which is the length of the $i+1$ row of $S$, and since every element of the $i+1$ row of $S$ is greater than or equal to $i+1$, we conclude that there is no column violation involving the rows $i$ and $i+1$. This holds for all $i=1,2,\dots,m-1$ and thus $S$ is semistandard.
			
			(2) Let $T \in \mathrm{SST}_{\al}(\mu)$. Since $\gamma_1 \ge \cdots \ge \gamma_{m-1} \ge 0$, it is clear that $T^{\vee}$ is semistandard. Hence we obtain the map $\mathrm{SST}_{\al}(\mu) \to \mathrm{SST}_{\al+\gamma}(\mu+\gamma), \ T \mapsto T^{\vee}$ which is injective.
			
			Let  $S \in \mathrm{SST}_{\al+\gamma}(\mu+\gamma)$. Since $S$ is semistandard, each element $i$ does not occur in rows $i+1,\dots,m$, for each $i=1,2,\dots,m-1$. By part (1) of the lemma, we may consider the tableau $T \in \mathrm{Tab}_\al(\mu)$ obtained from $S$ by deleting from the $i$th row, $\gamma_i$ appearances of the element $i$, for each $i=1,2,\dots,m-1$. Again by part (1) of the lemma, the $i$th row of $T$ contains the element $i$ at least $\mu_{i+1}$ times for each $i=1,2,\dots,m-1$ and moreover all the elements of the $m$-th row of $T$ are greater than $m-1$. Hence $T$ is standard. It is clear that $T^{\vee}=S$.
		\end{proof}
		
	Recall from Section 2.4 that for $A=(a_{ij}) \in M_{n}(\mathbb{N})$, we have the sequences $A^{(1)}\in \La(n)$ and $A^{(2)}\in \La(n)$ of column sums and row sums of $A$. Let us denote by $T_{n}(\mathbb{N})$ the set of $n \times n$ upper triangular matrices with entries in $\mathbb{N}$. For $\al, \beta  \in \La(n,r)$, define the set \[T_{n}(\mathbb{N})(\al, \beta) = \{A \in T_{n}(\mathbb{N}): A^{(1)}=\al, A^{(2)}=\beta\}.\]
	\begin{corollary}\label{triang}
		Let $\al \in P(\mu)$. Then a basis of the vector space $\Hom_G(D(\al),\Delta(\mu))$ is the set $\{\phi_{T_A}: A \in T_{n}(\mathbb{N})(\al, \mu)\}$.
	\end{corollary}
\begin{proof} We have the bijection (\ref{bijection1}), \[\mathrm{RSST}_{\al}(\mu) \to \{A \in M_{n}(\mathbb{N}): A^{(1)}=\al, A^{(2)}=\mu\}, \ T \mapsto A_T.\] If $T$ semistandard, then an element $i \in \{1,2,\dots,n\}$ cannot occur in a row of $T$ located below the $i$th row and hence $ A_T$ is upper triangular. By restricting the previous map we have the injective map \[\mathrm{SST}_{\al}(\mu) \to \{A \in T_{n}(\mathbb{N}): A^{(1)}=\al, A^{(2)}=\mu\}, \ T \mapsto A_T.\]
	This map is onto because if $B \in T_{n}(\mathbb{N})$ is such that $B^{(1)}=\al$ and $ B^{(2)}=\mu$, then the corresponding tableau $T_B$ has the property that an element $i \in \{1,2,\dots,n\}$ cannot occur in a row of $T_B$ located below the $i$th row, because $B$ is upper triangular. Since $\al \in P(\mu)$, we may apply  Lemma \ref{lem12}(1) for $\gamma =0$ to conclude that $T_B$ is semistandard. By the above bijection, the matrix of $T_B$ is $B$.
	\end{proof}

	\section{Homomorphisms and adding powers of $p$}
	\subsection{First main result and corollaries} Theorem \ref{thm13} that follows concerns pairs of partitions $\lambda, \mu \in \La^+(n,r)$ that satisfy (a) $\lambda \in P(\mu)$ and (b) an additional condition on $\mu$ that is given in the next definition.
	\begin{definition}\label{def}
		let $g$ be an integer such that $1 \le g \le n$. If $g=1$, let $\La^{+}(n)_1=\La^{+}(n)$ and $\La^{+}(n,r)_1=\La^{+}(n,r)$. If $g \ge 2$, let \[ \La^{+}(n)_g=\{\mu \in \La^{+}(n): \mu_{j-1} \le \mu_j +\mu_{j+1}, \ j=2,\dots,g\} \] and \[\La^{+}(n,r)_g= \La^{+}(n,r) \cap \La^{+}(n)_g.\]
	\end{definition}
	
	In the previous definition, it is understood that $\mu_{n+1}=0$.
\begin{example*}Let $n \ge4$. Then for $g=2$, we have that $\La^{+}(n)_2$ consists of the partitions $\mu=(\mu_1, \dots, \mu_n) \in \La^{+}(n)$ that satisfy \[\m_1 \le \mu_2+\mu_3. \]
For $g=3$ we have that $\La^{+}(n)_3$ consists of the partitions $\mu=(\mu_1, \dots, \mu_n) \in \La^{+}(n)$ that satisfy \begin{align*}
	&\mu_1 \le \mu_2+\mu_3,\\
	&\mu_2 \le \mu_3+\mu_4.
\end{align*}
In particular, $(7,6,3,2) \in \La^{+}(n)_2$ and $(7,6,3,2) \notin \La^{+}(n)_3$.
\end{example*}
From Definition \ref{def} it follows that we have the descending sequence \[\Lambda^+(n) = \Lambda^+(n)_1 \supseteq \Lambda^+(n)_2 \supseteq \cdots \supseteq  \Lambda^+(n)_n . \]
\begin{remark}\label{add}	We have  $\mu + \gamma \in \Lambda^+(n)_g$ if $\mu, \gamma \in \Lambda^+(n)_g.$ Indeed, if $\mu_{j-1} \le \mu_j +\mu_{j+1}$ and $\gamma_{j-1} \le \gamma_j +\gamma_{j+1}$ for $j=2, \dots, g$, then $\mu_{j-1} +\gamma_{j-1} \le (\mu_j +\gamma_j) + (\mu_{j+1} +\gamma_{j+1}).$\end{remark}
	We will need the following notation. \begin{definition}\label{defce} Let $\la, \mu \in \La^{+}(n,r)$.  \begin{enumerate}\item Define \[c_s= \sum_{i=1}^{s}(\mu_i - \lambda_i), \ s=1,\dots,n.\] \item Let $g \in \{2,\dots,n-1\}$. Define $e_s=e_s(\lambda, \mu, g) \in \mathbb{N}$ by
	\begin{equation}\label{e}
		e_s=\begin{cases} c_1, s=1,\\
			\max\{c_{s-1},c_s\}, \ 1<s<g,\\
			\min\{\lambda_{g+1}, c_g\}, \ s=g.
	\end{cases}\end{equation}
If $g=1$, define $e_1 =\min\{\lambda_{2}, c_1\}$. \end{enumerate}\end{definition}
	Our first main result is the following.
	\begin{theorem}\label{thm13}Let $\la, \mu \in \La^{+}(n,r)$ and  $\gamma \in \La^{+}(n)$. Suppose $\la \in P(\mu)$, $\mu \in \La^{+}(n)_g$ and $g<m$, where $g=\ell(\gamma), m=\ell(\mu)$. If $p^{\l_p(e_s)}$ divides $\gamma_s$ for all $s=1,\dots,g$, then \[\Hom_G(\Delta(\la),\Delta(\mu)) \simeq \Hom_{G}(\Delta(\la+\gamma), \Delta(\mu+\gamma)).\]
			\end{theorem}  
	We noted in the Introduction that the above theorem provides an answer to a problem  of D. Hemmer \cite[Problem 5.4]{Hem2}. 
	
	As corollaries we obtain a stability result and a periodicity result. As in the Introduction, if $k$ is a nonnegative integer and $\nu=(\nu_1,\dots,\nu_n)$ a partition, by $k\nu$ we denote the partition $(k\nu_1,\dots,k\nu_n).$ Let \begin{equation}\label{defdk}d_k=\dim \Hom _{G}(\Delta (\lambda + k \nu), \Delta(\mu + k \nu)).\end{equation}
	
\begin{corollary}\label{stab} Let $\la, \mu \in \La^{+}(n,r)$ and $\nu \in \La^{+}(n)$. Suppose $\la \in P(\mu)$ and  $\mu \in \La^{+}(n)_g$ and $g<m$, where $g=\ell(\nu), m=\ell(\mu)$. Then, the sequence $d_{p}, d_{p^2}, \dots$ 
eventually stabilizes.
\end{corollary}
\begin{proof} Indeed, according to Theorem \ref{thm13}, $\dim \Hom_{G}(\Delta (\lambda), \Delta(\mu)) = d_{p^N}$ for all $N \ge $ $\max\{e_1,\dots,e_g\}.$
\end{proof}		
\begin{corollary}\label{per} With the assumptions of the previous corollary, suppose in addition that $\nu \in \La^{+}(n)_g$. Then, the sequence $d_0, d_1, \dots$ 
	is periodic with period a power of $p$.
\end{corollary}
\begin{proof} From Lemma \ref{rem11}(4) and Remark \ref{add}, it follows that $\lambda + k \nu \in P(\mu)$ and $\mu + k \nu \in \Lambda^{+}(n)_g$ for all $k \in \mathbb{N}$. From  (\ref{e}) it is clear that $e_{i}(\lambda, \mu, g) = e_{i}(\lambda + k\nu, \mu + k\nu, g)$ for all $i=1,\dots,g$ and all $k \in \mathbb{N}$, because  $\ell(\nu) < g+1$. Thus we may apply Theorem \ref{thm13} to obtain \[\Hom_{G}(\Delta (\lambda + k \nu), \Delta(\mu + k \nu)) \simeq \Hom_{G}(\Delta (\lambda + (k+p^N) \nu), \Delta(\mu + (k+p^N) \nu)),\] where $N \ge $ $\max\{e_1,\dots,e_g\}.$
\end{proof}		
	\begin{examples}\label{rem14} (1) Let $p$ be a prime such that $p \ge 5$ and define the partitions \begin{align*}\lambda&=(4p, 3p-1, p-1, p-1, p-1, 4), \\ \mu &=(5p-1, 3p+1, 2p). \end{align*} The length of $\mu$ is $m=3$.  We have $\mu_2 \le \lambda_1$ and $\mu_1+\mu_3 \le \lambda_1 + \lambda_2$ and thus $\lambda \in P(\mu)$. Let $g=2$. We have $\mu_1 \le \mu_2+\mu_3$ and thus $\mu \in \La^+(n)_g$. Using eq. (\ref{e}) we see that \[
			e_1 = c_1=p-1, \ e_2 = \min\{\lambda_3, c_2\}=p-1,\  
			\]
	and thus $l_p(e_1)=\l_p(e_2)=1$. Hence Theorem \ref{thm13} states that $\Hom_G(\Delta(\la),\Delta(\mu)) \simeq \Hom_{G}(\Delta(\la+\gamma), \Delta(\mu+\gamma))$ for all partitions $\gamma=(\gamma_1, \gamma_2)$ such that $\gamma_1$ and $\gamma_2$ are divisible by $p$ and $\gamma_2 \neq 0$.
	
	For $g=1$ in the above example, Theorem \ref{thm13} states that $\Hom_G(\Delta(\la),\Delta(\mu)) \simeq \Hom_{G}(\Delta(\la+(\gamma_1)), \Delta(\mu+(\gamma_1))$ for all nonnegative $\gamma_1$ that are divisible by $p$. Hence the final conclusion is that $\Hom_G(\Delta(\la),\Delta(\mu)) \simeq \Hom_{G}(\Delta(\la+\gamma), \Delta(\mu+\gamma))$ for all partitions $\gamma=(\gamma_1, \gamma_2)$ such that $\gamma_1$ and $\gamma_2$ are divisible by $p$.
	
	We note that in this example we have  $\Hom_{G}(\Delta(\la), \Delta(\mu) \neq 0$ according to Theorem \ref{nonv1} that is proved in Section 5 below. See Example \ref{exhom}.
	
	(2)	Let us fix a partition $\mu$ such that $\mu \in \La^+(n,r)_{m-1}$, 	where $m=\ell{(\mu) }$, and $\mu$ has distinct parts. Let $q$ be an integer such that \begin{equation}\label{defq}0 < q \le \min\{\mu_1 - \mu_2, \mu_2 - \mu_3, \dots, \mu_{m-1} - \mu_{m}, \mu_{m}\}.\end{equation}
	Define the partition \begin{equation}\label{deflambda}\lambda =(\m_1-q,\mu_2, \dots, \mu_m, q).\end{equation} Using (\ref{defq}) it follows that $\lambda \in P(\mu)$. From eq. (\ref{e}) we deduce \[e_1 = \dots = e_{m-1}=q.\]
	Let $\nu$ be any partition of length at most $m-1$. Applying Corollary \ref{stab} and its proof, we have that the sequence of nonnegative integers $d_p, d_{p^2}, \dots$ defined in eq. (\ref{defdk}) stabilizes at $\label{eqd}d_{p^{l_p(q)}}$, that is \begin{equation}d_{p^{l_p(q)}} = d_{p^{l_p(q)+1}}= \dots \end{equation}
	and this common value is equal to $\dim \Hom_G(\Delta(\la),\Delta(\mu)).$
	
	We now relate this example to a result of Carter and Payne.
	
	If $p$ be a prime and $\lambda, \mu \in \La^+(n,r)$, we say that $\lambda, \mu$ are a pair of Carter-Payne partitions if there exist $1 \le i<j \le n$ and $q>0$ such that \begin{equation}\label{CPpar}\mu_i=\lambda_i +q, \ \mu_j=\lambda_j -q, \ \mu_u=\lambda_u, \ u \neq i,j \end{equation}
	and \begin{equation}\label{CPdiv} p^{l_p(q)} \ \textrm{divides} \ \lambda_i-\lambda_j +j-i+q.
	\end{equation}
	We recall the following classical result of Carter and Payne.
	\begin{theorem}[{\cite{CP}}]\label{CPthm}
If $\lambda, \mu \in \La^+(n,r)$ are a pair of Carter-Payne partitions, then \[\Hom_G(\Delta(\lambda), \Delta(\mu)) \neq 0.\]
	\end{theorem}
	The dimensions of the hom spaces $\Hom_G(\Delta(\lambda), \Delta(\mu))$, where $\lambda, \mu$ are a pair of Carter-Payne partitions, are not known in general to the best of our knowledge. They are known to be equal to 1 when $q=1$ in (\ref{CPpar}) and $p>2$, \cite{EM}.
	
	Now let us consider again partitions $\lambda,  \mu $ as  in (\ref{deflambda}), where $q$ satisfies (\ref{defq}), $\mu \in \La^+(n,r)_{m-1}$ and $\mu$ has distinct parts. Suppose in addition that $p^{l_p(q)}$ divides $\mu_1-q+m$. Then $\lambda, \mu$ are a pair of Carter-Payne partitions (for $i=1$ and $j=m+1$) and by Theorem \ref{CPthm} we have $\Hom_G(\Delta(\lambda), \Delta(\mu)) \neq 0$. On easily checks that $\lambda +p^t\nu, \mu +p^t\nu$ are a pair of Carter-Payne partitions (for $i=1$ and $j=m+1$) for every $t \ge l_p(q)$. Hence the dimensions $d_{p^t}$ are nonzero for all $ t \ge l_p(q)$. The new information provided by Corollary \ref{stab} is that these dimensions are in fact equal and the common value is equal to $\dim \Hom_G(\Delta(\la),\Delta(\mu)).$
	
	(3) Theorem \ref{thm13} may be useful in computing dimensions of hom spaces between Weyl modules for particular cases. For example, let $p=3$ and \begin{align*}\lambda&=(11, 10, 7, 3, 3), \\ \mu &=(14, 10, 7, 3). \end{align*} For $g=3$ one checks that  $e_1=e_2=e_3=3$ and thus $l_p(e_i)=2, \ i=1,2,3$. From Theorem \ref{thm13} we conclude that $\Hom_G(\Delta(\la),\Delta(\mu)) \simeq \Hom_{G}(\Delta(\la+\gamma), \Delta(\mu+\gamma))$ for any partition $\gamma = (\gamma_1, \gamma_2, \gamma_3)$ such that every $\gamma_i$ is divisible by $3^2$. Using the GAP4 program written by M. Fayers \cite{F}, one finds $\dim\Hom_G(\Delta(\lambda), \Delta(\mu))=2.$ Hence \[\dim\Hom_G(\Delta(\lambda +\gamma), \Delta(\mu + \gamma))=2\] for all  $\gamma$ as above.

		(4) Next, we show by examples that none of the hypotheses $\la \in P(\mu)$ and $ \mu \in \La^{+}(n)_g$ of Theorem \ref{thm13} may be omitted.
		
		\ \ (a) Let $p=3$, $\la=(4,3,2,2)$, $\mu=(5,5,1)$ and $\gamma=(6, 3)$. Here, $g=2<3=m$ and $e_1=1, e_2=2$. Also $\la \notin P(\mu)$ and $\mu \in \La^{+}(n)_g$. Using \cite{F}, one finds  \begin{center}
			$\Hom_G(\Delta(\la),\Delta(\mu))=0$ and $\Hom_{G}(\Delta(\la+\gamma), \Delta(\mu+\gamma)) \neq 0$.
		\end{center}
		
		\ \ (b) Let $p=2$, $\lambda = (5,4,1,1)$, $\mu=(8,2,1)$ and $\gamma = (4,2)$. Here, $g=2<3=m$ and $e_1=3, e_2=1$. Also $\la \in P(\mu)$ and $\mu \notin \La^{+}(n)_g$. Using \cite{F} one finds \begin{center}
			$\Hom_G(\Delta(\la),\Delta(\mu)) \neq 0$ and $\Hom_{G}(\Delta(\la+\gamma), \Delta(\mu+\gamma)) = 0$.
		\end{center}
	\end{examples}

	\subsection{Proof of Theorem \ref{thm13}.} \label{s4.2} We note that $\lambda \unlhd \mu$ if and only if $\lambda +\gamma \unlhd \mu + \gamma$. Since $\Delta{(\lambda)}$ is a cyclic $G$-module generated by an element of weight $\lambda$ and since every weight $\alpha$ of $\Delta(\mu)$ satisfies $\alpha \unlhd \mu$ (\cite[II 2.2 Prop]{Jan}), it follows that both Hom spaces are zero if $\lambda \ntrianglelefteq \mu$ and thus we assume $\lambda \unlhd \mu$.

We need some notation. If $\al=(\al_1,\dots,\al_n) \in \La(n,r)$ and $s, t$ are integers such that $1 \le s \le n-1$ and $1 \le t \le \al_{s+1} $, let us denote the sequence $(\al_1,\dots,\al_s+t, \al_{s+1}-t,\dots,\al_n) \in \La(n,r)$ by $\al({s,t})$, \[\al({s,t}) = (\al_1,\dots,\al_s+t, \al_{s+1}-t,\dots,\al_n). \] Also let \[\al^{\vee} =\al+\gamma.\]

Taking into account the presentations of $\Delta(\mu)$ and $\Delta(\mu ^{\vee})$ from Section 2.4, consider for each $s=1,\dots,n-1$ the following diagram,

\begin{equation}\label{diag1} 
	\begin{tikzcd}
		\Hom_G(D(\lambda),\Delta(\mu))\arrow{d}{\psi_{\la}} \arrow{rrrr}{ \pi_s \circ \Hom_G(\square_{\la},\Delta({\mu}))}
		&&&& \sum_{t=1}^{\lambda_{s+1}}\Hom_G(D(\la({s,t})),\Delta(\mu)) \arrow{d}{\sum_{t=1}^{\lambda_{s+1}}{\psi}_{\la(s,t)}} \\ \arrow{rrrr}{ \pi^{\vee}_s \circ \Hom_{G}(\square_{\la^{\vee}},\Delta({\mu^{\vee}}))}
		\Hom_{G}(D(\lambda^{\vee}),\Delta(\mu^{\vee})) 
		&&&& \sum_{t=1}^{\lambda^{\vee}_{s+1}}\Hom_{G}(D(\la^{\vee}({s,t})),\Delta(\mu^{\vee}))
	\end{tikzcd}
\end{equation}
where $\psi_{\la}$ and $\psi_{\la(s,t)}$ are the isomorphisms from Lemma \ref{lem12}(2) and the restriction of the right vertical map on the summand  $\Hom_G(D(\la({s,t})),\Delta(\mu))$ is the map \[\psi_{\la(s,t)} : \Hom_G(D(\la({s,t})),\Delta(\mu)) \to \Hom_{G}(D(\la^{\vee}({s,t})),\Delta(\mu^{\vee})),\] where $1 \le t \le t_{s+1}$. Also, $\pi_{s}$ and $\pi^{\vee}_{s}$ are the indicated natural projections \[\sum_{s=1}^{n-1}\sum_{t=1}^{\lambda_{s+1}}\Hom_G(D(\la(s,t)),\Delta(\mu)) \to \sum_{t=1}^{\lambda_{s+1}}\Hom_G(D(\la(s,t)),\Delta(\mu)),\] \[\sum_{s=1}^{n-1}\sum_{t=1}^{\lambda_{s+1}}\Hom_G(D(\la^{\vee}(s,t)),\Delta(\mu)) \to \sum_{t=1}^{\lambda_{s+1}}\Hom_G(D(\la^{\vee}(s,t)),\Delta(\mu)),\] respectively. We note that the vertical map on the right in diagram (\ref{diag1}) is in general a monomorphism since the number of summands in the codomain is equal to $\la^{\vee}_{s+1}$, the number f summands in the domain is equal to  $\la_{s+1}$ and we have $\la_{s+1} \le \la_{s+1} + \gamma_{s+1}= \la^{\vee}_{s+1}$. We intend to show that diagram (\ref{diag1}) is commutative for all $s$. If this is the case, then by taking kernels of the horizontal maps of diagram (\ref{diag1}) we obtain the commutative diagram
\begin{equation}\label{diag2} 
	\begin{tikzcd}[scale cd=0.94]
	Z \arrow[r, hook] \arrow[d] &[-1em]	\Hom_G(D(\lambda),\Delta(\mu))\arrow{d}{\psi_{\la}} \arrow{rrrr}{ \pi_s \circ \Hom_G(\square_{\la},\Delta({\mu}))}
		&&&&[-1em]\sum_{t=1}^{\lambda_{s+1}}\Hom_G(D(\la({s,t})),\Delta(\mu)) \arrow{d}{\sum_{t=1}^{\lambda_{s+1}}{\psi}_{\la(s,t)}} \\ Z'  \arrow[r, hook] & \arrow{rrrr}{ \pi^{\vee}_s \circ \Hom_{G}(\square_{\la^{\vee}},\Delta({\mu^{\vee}}))}
		\Hom_{G}(D(\lambda^{\vee}),\Delta(\mu^{\vee})) 
		&&&& \sum_{t=1}^{\lambda^{\vee}_{s+1}}\Hom_{G}(D(\la^{\vee}({s,t})),\Delta(\mu^{\vee}))
	\end{tikzcd}
\end{equation}
where $Z=\Hom_G(\Delta(\lambda),\Delta(\mu)), Z'=\Hom_G(\Delta(\lambda^{\vee}),\Delta(\mu^{\vee}))$ and the left vertical map is the restriction of $\psi_{\lambda}$ to $Z$. In diagram (\ref{diag2}) the map $\psi_{\lambda}$ is an isomorphism and the map  $\sum_{t=1}^{\lambda_{s+1}}{\psi}_{\la(s,t)}$ is monomorphism as we noted before. By applying the Five Lemma \cite[Proposition 2.72(i),(ii)]{Ro}, we will have $Z \simeq Z'$, that is $\Hom_G(\Delta(\la),\Delta(\mu)) \simeq$ $ \Hom_{G}(\Delta(\la+\gamma), \Delta(\mu+\gamma))$.

We now show that diagram (\ref{diag1}) is commutative. Fix $s$ such that $1 \le s \le n-1$. Let $ T \in \st$ with corresponding matrix $A=(a_{ij})$. Since $T$ is semistandard, $A$ is upper triangular. In particular, $a_{i, s+1}=0$ for all $i>s+1$. Using this and Corollary \ref{corhom}, we see that the image of $\phi_T$ under the top horizontal map of diagram (\ref{diag1}) is equal to \begin{equation}\label{eq111}
	\sum_{t=1}^{\lambda_{s+1}}\sum_{t_1,\dots,t_{s+1}}\tbinom{a_{1s}+t_1}{t_1}\tbinom{a_{2s}+t_2}{t_2} \cdots \tbinom{a_{ss}+t_s}{t_s}\phi_{T(s,t;t_1,\dots,t_{s+1},0,\dots,0)},\end{equation}
	where the right sum is over all  $t_1,\dots,t_{s+1} \in \mathbb{N} $ such that \begin{equation}\label{ineq1111}t_{1}+\dots+t_{s+1}=t,\  a_{1,s+1}-t_1 \ge 0, \ \dots, \ a_{s+1,s+1}-t_{s+1} \ge 0.\end{equation}
We note that the tableau $T(s,t;t_1,\dots,t_{s+1},0,\dots,0)$ has weight $\la(s,t)$.

For the counterclockwise direction we note that the matrix $(a^{\vee}_{ij})$ of the tableau $T^{\vee} \in \mathrm{SST}_{\la+\al}(\mu+\al)$ is given by $a^{\vee}_{ii} = a_{ii}+\gamma_i$ for all $i$ and $a^{\vee}_{ij} = a_{ij}$ for all $i \neq j$. According to Corollary \ref{corhom}, the image of $\phi_T$ under the counterclockwise direction of diagram (\ref{diag1}) is equal to \begin{equation}\label{eq211}
	\sum_{t=1}^{\lambda_{s+1}+\gamma_{s+1}}\sum_{u_1,\dots,u_{s+1}}\tbinom{a_{1s}+u_1}{u_1}\tbinom{a_{2s}+u_2}{u_2} \cdots \tbinom{a_{ss}+\gamma_s+u_s}{u_s}\phi_{T^{\vee}(s,u;u_1,\dots,u_{s+1},0,\dots,0)},\end{equation}
where the right sum is over all $u_1,\dots,u_{s+1} \in \mathbb{N}$ such that \begin{align*}&u_{1}+\dots+u_{s+1}=u, \  a_{1,s+1}-u_1 \ge 0, \  \dots, \  a_{s,s+1}-u_s \ge 0, \\ &a_{s+1,s+1}+\gamma_{s+1}-u_{s+1} \ge 0.\end{align*}

Recall that we are assuming $g<m$. We distinguish three cases according to the relative size of $s.$ 

\textbf{Case 1.} Suppose $g<m \le s$. Then from the definition we have \begin{equation}\label{eq13}T(s,t;t_1,\dots,t_{s+1},0,\dots,0)=\begin{matrix*}[l]
		1^{(a_{11})} \cdots s^{(a_{1s}+t_1)}(s+1)^{(a_{2,s+1}-t_1)} \cdots n^{(a_{1n})} \\
		\cdots \\ 
		m^{(a_{mm})}\cdots s^{(a_{ms}+t_m)}(s+1)^{(a_{m,s+1}-t_m)} \cdots n^{(a_{mn})}   \end{matrix*}.\end{equation}
The weight of this tableau is $\la(s,t)$ and thus Lemma \ref{rem11}(3) yields $\la(s,t) \in P(\mu)$. Hence we may apply Lemma \ref{lem12} for $\la(s,t)$ in place of $\al$ and for $\gamma=0$ to conclude from (\ref{eq13}) that $T(s,t;t_1,\dots,t_{s+1},0,\dots,0)$ is semistandard. Thus we may apply the maps of Lemma \ref{lem12}(3) to  (\ref{eq111}) and we obtain that the image of $\phi_T$  in the clockwise direction of diagram (\ref{diag1}) is equal to 
\begin{equation}\label{eq14}
	\sum_{t=1}^{\lambda_{s+1}}\sum_{t_1,\dots,t_{s+1}}\tbinom{a_{1s}+t_1}{t_1}\tbinom{a_{2s}+t_2}{t_2} \cdots \tbinom{a_{ss}+t_s}{t_s}\phi_{T(s,t;t_1,\dots,t_{s+1},0,\dots,0)^{\vee}},\end{equation}
where the right sum is over conditions (\ref{ineq1111}).
By assumption we have $g < s$ and hence $\gamma_s=\gamma_{s+1}=0 $ and $\la^{\vee}_{s+1}=\la_{s+1}$. Thus $T(s,t;t_1,\dots,t_{s+1},0,\dots,0)^{\vee}$ $=T^{\vee}(s,t;t_1,\dots,t_{s+1},0,\dots,0)$ and (\ref{eq14}) and (\ref{eq211}) are equal.

\textbf{Case 2.} Suppose $g<s<m.$ As before, the image of $\phi_T$ under the top horizontal map of diagram (\ref{diag1}) is equal to (\ref{eq11}). Since $s<m$, we have \begin{equation}T= \begin{matrix*}[l]
		1^{(a_{11})} \cdots n^{(a_{1n})} \\
		\cdots \\ 
		s^{(a_{ss})} \cdots n^{(a_{sn})} \\
		(s+1)^{(a_{s+1,s+1})} \cdots n^{(a_{s+1,n})}\\
		\cdots\\
		m^{(a_{mm})} \cdots n^{(a_{mn})}  \end{matrix*}\end{equation}
and thus 
\begin{equation*}\label{case21}T(s,t;t_1,\dots,t_{s+1},0,\dots,0)= \begin{matrix*}[l]
		1^{(a_{11})} \cdots s^{(a_{1s}+t_1)}(s+1)^{(a_{1s+1}-t_1)} \cdots n^{(a_{1n})} \\
		\cdots \\ 
		s^{(a_{s,s}+t_s)}(s+1)^{(a_{s,s+1}-t_s)} \cdots n^{(a_{s,n})} \\
		s^{(t_{s+1})}(s+1)^{(a_{s+1,s+1}-t_{s+1})} \cdots n^{(a_{s+1,n})}\\
		(s+2)^{(a_{s+2,s+2})}\cdots n^{(a_{s+2,,n})} \\
		\cdots\\
		m^{(a_{mm})} \cdots n^{(a_{mn})}  \end{matrix*}.\end{equation*}
	This last tableau is not in general semistandard because of the appearance of $s^{(t_{s+1})}$ in the $s+1$ row. Consider the tableau \begin{equation}\label{eq117}T(s,t;t_1,\dots,t_{s+1},0,\dots,0)[s,s+1]=\begin{matrix*}[l] 
			s^{(a_{ss}+t_s)}(s+1)^{(a_{s,s+1}-t_s)} \cdots n^{(a_{sn})} \\
			s^{(t_{s+1})}(s+1)^{(a_{s+1,s+1}-t_{s+1})} \cdots n^{(a_{s+1,n})} \end{matrix*} \end{equation} consisting of rows $s$ and $s+1$ of $T(s,t;t_1,\dots,t_{s+1},0,\dots,0)$. If $a_{ss}+t_s+t_{s+1} > \mu_s$, then $T(s,t;t_1,\dots,t_{s+1},0,\dots,0)[s,s+1]=0$ according to Lemma \ref{lem15}(1). So suppose \begin{equation}\label{case2n1}a_{ss}+t_s+t_{s+1} \le \mu_s.\end{equation} Applying Lemma \ref{lem15}(2) we have that $[T(s,t;t_1,\dots,t_{s+1},0,\dots,0)[s,s+1]]$ is equal to \begin{align*}
		&(-1)^{t_{s+1}}\sum_{k_{s+1},\dots,k_n}\tbinom{a_{s+1s+1}-t_{s+1}+k_{s+1}}{k_{s+1}}\cdots\tbinom{a_{s+1n}+k_n}{k_n}\\
		&\begin{bmatrix*}[l] 
			s^{(a_{ss}+t_s+t_{s+1})}(s+1)^{(a_{s,s+1}-t_s-k_{s+1})} \cdots n^{(a_{s,n-k_n})} \\
			(s+1)^{(a_{s+1,s+1}-t_{s+1}+k_{s+1})} \cdots n^{(a_{s+1,n}+k_n)} \end{bmatrix*},
	\end{align*} 
	where the sum ranges over all
	$k_{s+1},\dots,k_n \in \mathbb{N}$ such that \begin{equation}\label{case23}k_{s+1}+\dots+k_n=t_{s+1}, \  k_{s+1} \le a_{s,s+1}-t_s, \  k_{s+2} \le a_{s,s+2}, \  \dots, \  k_n \le a_{sn}.\end{equation}
	Upon substitution in (\ref{eq111}) according to Lemma \ref{insertrows} we obtain
	\begin{align*}
		&\sum_{t=1}^{\lambda_{s+1}}\sum_{t_1,\dots,t_{s+1}}\tbinom{a_{1s}+t_1}{t_1}\tbinom{a_{2s}+t_2}{t_2} \cdots \tbinom{a_{ss}+t_s}{t_s}(-1)^{t_{s+1}}\\& \sum_{k_{s+1},\dots,k_n}\tbinom{a_{s+1,s+1}-t_{s+1}+k_{s+1}}{k_{s+1}}\cdots\tbinom{a_{s+1,n}+k_n}{k_n}\phi_{T(s,t;t_1,\dots,t_{s+1},0,\dots,0)_{k_{s+1},\dots,k_n}},
	\end{align*} 
	where  \begin{align}\label{eq1110}T(s,t&;t_1,\dots,t_{s+1},0,\dots,0)_{k_{s+1},\dots,k_n}= \\\nonumber&\begin{matrix*}[l]
			1^{(a_{11})} \cdots s^{(a_{1s}+t_{1})}(s+1)^{(a_{1,s+1}-t_1)} \cdots n^{(a_{1n})} \\
			\vdots \\ 
			(s-1)^{(a_{s-1,s-1})}s^{(a_{s-1,s}+t_{s-1})}(s+1)^{(a_{s-1,s+1}-t_{s-1})} \cdots n^{(a_{s-1,n})} \\
			s^{(a_{ss}+t_s+t_{s+1})}(s+1)^{(a_{s,s+1}-t_s-k_{s+1})} \cdots n^{(a_{s,n-k_n})} \\
		(s+1)^{(a_{s+1,s+1}-t_{s+1}+k_{s+1})} \cdots n^{(a_{s+1,n}+k_n)}\\
			(s+2)^{(a_{s+2,s+2})}\cdots n^{(a_{s+2,n})} \\
			\vdots\\
			m^{(a_{mm})} \cdots n^{(a_{mn})}  \end{matrix*},\end{align} the middle sum is over all  $t_1,\dots,t_{s+1} \in \mathbb{N}$ subject to (\ref{ineq1111}), (\ref{case2n1}), and the right sum is over all  $k_{s+1},\dots,k_{n} \in \mathbb{N}$ subject to (\ref{case23}). 
	
	The weight of each tableau in (\ref{eq1110}) is equal to $\la(s,t)$ and by Lemma \ref{rem11}(3) we have $\la(s,t) \in P(\mu)$. Moreover, each tableau in (\ref{eq1110}) is row semistandard and has the property that for each $i=1,\dots,n$ all its entries equal to $i$ are contained in the first $i$ rows.  Hence we may apply Lemma \ref{lem12}(1) for $\la(s,t)$ in place of $\al$ and for $\gamma=0$ to conclude that $T(s,t;t_1,\dots,t_{s+1},0,\dots,0)_{k_{s+1},\dots,k_n}$ is semistandard. Thus we may apply the maps of Lemma \ref{lem12}(3) to obtain that the image of $\phi_T$ in the clockwise direction of diagram (\ref{diag1}) is equal to 
	\begin{align}\label{case25}
		&\sum_{t=1}^{\lambda_{s+1}}\sum_{\substack{t_1,\dots,t_{s+1}}}\tbinom{a_{1s}+t_1}{t_1}\tbinom{a_{2s}+t_2}{t_2} \cdots \tbinom{a_{ss}+t_s}{t_s}(-1)^{t_{s+1}}\\& \sum_{k_{s+1},\dots,k_n}\tbinom{a_{s+1,s+1}-t_{s+1}+k_{s+1}}{k_{s+1}}\cdots\tbinom{a_{s+1,n}+k_n}{k_n}\phi_{(T(s,t;t_1,\dots,t_{s+1},0,\dots,0)_{k_{s+1},\dots,k_n})^{\vee}}\nonumber,
	\end{align} where the sums are over the same conditions as in the previous paragraph.
	
A similar computation shows that the image of $\phi_T$ in the counterclockwise direction of diagram (\ref{diag1}) is equal to \begin{align}\label{case26}
	&\sum_{u=1}^{\lambda_{s+1}+\gamma_{s+1}}\sum_{u_1,\dots,u_{s+1}}\tbinom{a_{1s}+u_1}{u_1}\tbinom{a_{2s}+u_2}{u_2} \cdots \tbinom{a_{ss} + \gamma_s + u_s}{u_s}(-1)^{u_{s+1}}\\& \sum_{l_{s+1},\dots,l_n}\tbinom{a_{s+1,s+1}+\gamma_{s+1}-u_{s+1}+l_{s+1}}{l_{s+1}}\cdots\tbinom{a_{s+1,n}+l_n}{l_n}\phi_{T^{\vee}(s,u;u_1,\dots,u_{s+1},0,\dots,0)_{l_{s+1},\dots,l_n}}\nonumber,
\end{align} where the middle sum is over all  $u_1,\dots,u_{s+1} \in \mathbb{N}$ such that \begin{align}\label{case27}&u_{1}+\dots+u_{s+1}=u,\\\label{case28}&a_{1s+1}-u_1 \ge 0, \ \dots, \ a_{s,s+1}-u_{s} \ge 0, \ a_{s+1,s+1}+\gamma_{s+1}-u_{s+1} \ge 0, \\\label{case29}&a_{ss} + \gamma_s +u_s+u_{s+1}  \le \mu_s +\gamma_s,\end{align}
and the right sum is over all 
$l_{s+1},\dots, l_n \in \mathbb{N}$ such that \begin{equation}\label{case211}l_{s+1}+\dots+l_n=u_{s+1}, \  l_{s+1} \le a_{s,s+1}-u_s, \  l_{s+2} \le a_{s,s+2}, \  \dots, \  l_n \le a_{sn}.\end{equation} 
By assumption we have $g<s$ and thus in particular $\gamma_s = \gamma_{s+1}=0$. Hence (\ref{case25}) and (\ref{case26}) are equal.

So far we have not used the divisibility hypotheses of the theorem. Recall the notation $c_s= \sum_{i=1}^{s}(\mu_i - \lambda_i)$,\ $s=1,\dots,n$.

\textbf{Case 3.} Suppose $1 \le s \le g.$ Exactly as in Case 2, the image of $\phi_T$ in the clockwise direction of diagram (\ref{diag1}) is equal to (\ref{case25}), where the middle and right sums are over all  $t_1,\dots,t_{s+1},k_{s+1},\dots,k_n \in \mathbb{N}$ subject to (\ref{ineq1111}), (\ref{eq117}), (\ref{case23}), and the image of $\phi_T$ in the counterclockwise direction of diagram (\ref{diag1}) is equal to (\ref{case26}), where the middle and right sums are over all  $u_1,\dots,u_{s+1},l_{s+1},\dots,l_n \in \mathbb{N}$ subject to (\ref{case27}) - (\ref{case211}).

Claim 1. We claim that in the left sum of (\ref{case25}), we may assume that $t \le c_s$ if $s<g$ and that $t \le \min\{\lambda_{g+1}, c_g\}$ if $s=g$. Indeed, every tableau in  (\ref{case25}) has the property that all elements $1,\dots,s$ are located in the first $s$ rows and thus $t \le c_s$ by Remark \ref{obv}. By Lemma \ref{rem11}(1), we have $\la_{s+1} \ge c_s$ if $s<g$. Hence we may assume in (\ref{case25}) that $t \le c_s$ if $s<g$. If $s=g$, then by Remark \ref{obv}, $t \le c_g$ and we may assume in (\ref{case25}) that $t \le \min\{\lambda_{g+1}, c_g\}$

Claim 2. We claim that in the left sum of (\ref{case26}), we may assume that $u \le c_s$ if $s<g$ and that $u \le \min\{\lambda_{g+1}, c_g\}$ if $s=g$. Indeed, the proof is almost identical to the previous proof (the $\gamma$'s cancel out) except when $s=g$ in which case $\gamma_{s+1}=0$ by assumption of the Theorem \ref{thm13}.
 
 Claim 3. Consider the last inequality in (\ref{case28}), that is $a_{s+1,s+1}+\gamma_{s+1}-u_{s+1} \ge 0$.  We claim that in the middle sum of (\ref{case26}), we may assume $a_{s+1,s+1}-u_{s+1} \ge 0$. Indeed, this is clear if $g=s$, since $\gamma_{g+1}=0.$ So let $s<g$, whence $s<g<m$ by the assumption of the theorem. Suppose  $a_{s+1,s+1}-u_{s+1}<0$.  Then from inequality (\ref{case29}) we have
 \[\mu_s \ge a_{ss}+u_s+u_{s+1} > a_{ss}+u_s+a_{s+1, s+1} \ge a_{ss}+a_{s+1, s+1}.\]
 Since $s+2 \le m$, we may apply Lemma \ref{lem12}(1) (for $\alpha = \la, \gamma = 0$) to conclude that  $a_{ss} \ge \mu_{s+1}$ and $a_{s+1,s+1} \ge \mu_{s+2}$. Hence $\mu_s > \mu_{s+1}+\mu_{s+2}$ which contradicts the hypothesis $\mu \in \La^{+}(n,r)_g$.
 
 Claim 4. We have the equalities in $K$ 
 \begin{equation*}\label{eq115}\tbinom{a_{ss} + \gamma_s + u_s}{u_s}=\tbinom{a_{ss}  + u_s}{u_s}, \  \tbinom{a_{s+1,s+1}+\gamma_{s+1}-u_{s+1}+l_{s+1}}{l_{s+1}}=\tbinom{a_{s+1,s+1}-u_{s+1}+l_{s+1}}{l_{s+1}}.\end{equation*}
 Indeed, from Claim 2 we have  $u \le c_s$ for all $s$ and thus $u_s \le c_s$. From the hypothesis of the theorem, $p^{l_p(u_s)}$ divides $\gamma_s$ which by Lemma \ref{bin1}(1) implies the first equality of Claim 4. For the second equality, we note it holds if $s=g$, since $\gamma_{g+1}=0$. Suppose $s<g$. From Claim 2 we have  $u \le c_s$ and thus from (\ref{case211}) and (\ref{case27}), $l_{s+1} \le c_s$. So from the hypothesis of the Theorem \ref{thm13}, $p^{l_p(l_{s+1})}$ divides $\gamma_{s+1}$. From this and Claim 3, it follows we may apply Lemma \ref{bin1}(1) to get the second equality of the claim.
 
 Now from Claims 1-4 it follows that (\ref{case25}) and (\ref{case26}) are equal. Thus diagram (\ref{diag1}) commutes in all three cases.
 
 The proof of Theorem \ref{thm13} is complete. \qed

\begin{remark} Suppose $\gamma=(\gamma_1)$ is a partition consisting of a singe part. In \cite{MS4} it was shown that the extension groups $\Ext^i_G(\Delta(\lambda), \Delta(\mu))$ and $\Ext^i_G(\Delta(\lambda + \gamma), \Delta(\mu +\gamma))$ are isomorphic as vector spaces under suitable conditions on $p, \lambda, \mu, \gamma_1$. Let us comment here how the special case of Theorem \ref{thm13} of the present paper for $\gamma = (\gamma_1)$ relates to the above result for $i=0$. 
	
	Suppose in the statement of Theorem \ref{thm13} we have $g=1$, so that $\gamma=(\gamma_1)$. According to Definition \ref{def} we have $\La^+(n)_1=\La^+(n)$ and according to the definition given in eq. (\ref{e}) we have \[e_1 =\min\{\lambda_{2}, \mu_1-\lambda_1\}.\] Thus from Theorem \ref{thm13} we conclude that if \begin{center}
		$\lambda \in P(\mu)$ and $\gamma_1 =p^d$, where $d > \min\{\lambda_{2}, \mu_1-\lambda_1\}$,
	\end{center} then $\Hom_G(\Delta(\la),\Delta(\mu)) = \Hom_{G}(\Delta(\la+(\gamma_1)), \Delta(\mu+(\gamma_1)))$. This result was obtained in \cite[Theorem 6.1]{MS4} under the weaker assumption $\mu_2 \le \lambda_1$ in place of $\lambda \in P(\mu)$.  However, let us mention without providing a proof here, that for $g=1$ the proof of Theorem \ref{thm13} is considerably simpler and in fact one may verify that it reduces to the proof of \cite[Theorem 6.1]{MS4}. In particular, one of the simplifications that occurs in this special case, is that we only need the first inequality $\mu_2 \le  \la_1$ of the inequalities in the definition of $\la \in P(\mu)$, see Definition \ref{defP}.
	\end{remark}
	
\section{A non vanishing result}
 
\subsection{Second main result}\label{s5.1}Let us fix a pair of partitions $\la, \mu \in \La^{+}(n,r)$ such that $\lambda \unlhd \mu$ and denote by $m$ the length of $\mu$. Recall the notation $c_s=\sum_{i=1}^{s}(\mu_i - \lambda_i)$ and the definition of $l_p(a)$ for a positive integer $a$ from Section 2.5. We let \begin{center}
	$c'_{m-1}=\min\{c_{m-1}, \lambda_{m}\}$ and $l_p(0)=0$. 
\end{center}

\begin{definition}\label{defpsi}Let $\psi$ be the map \[\psi : D(\la) \to \Delta(\mu), \ \psi = \sum_{T \in \mathrm{SST}_{\la}(\mu)}\phi_T \] corresponding to the sum of all semistandard tableaux of weight $\lambda$ and shape $\mu$. \end{definition} Our second main result  is the following.
\begin{theorem}\label{nonv1}
	Let $\la, \mu \in \La^{+}(n,r)$ such that $\la \unlhd \mu$ and $\la \in P(\mu)$. If 
		\begin{enumerate}
		\item $p^{l_p(c_s)}$ divides $\la_s-\mu_{s+1}+1$ for all $s=1,\dots,m-2$  and $p^{l_p(c'_{m-1})}$ divides $\la_{m-1}-\mu_{m}+1$, and
		\item $p^{l_p(\lambda_{s+1})}$ divides $\la_s+1$ for all $s=m,\dots,n-1$,
	\end{enumerate}
	then $\psi$ induces a nonzero homomorphism $\Delta(\la) \to \Delta(\mu)$.
\end{theorem}
	\begin{example}\label{exhom} Let us recall the partitions in Example \ref{rem14}(1). Let $p$ be a prime such that $p \ge 5$ and let \begin{align*}\lambda&=(4p, 3p-1, p-1, p-1, p-1, 4), \\ \mu &=(5p-1, 3p+1, 2p). \end{align*}
	We observed that $\lambda \in P(\mu)$. One checks that $\la \unlhd \mu$. With the notation of Theorem \ref{nonv1} we have \[c_1=p-1, \ c_2' =\min\{p+1, p-1\}=p-1. \] Thus $l_p(c_1)=l_p(c'_2)=1$. Also $\lambda_1-\mu_2+1=\lambda_2 - \mu_3 + 1=p$ are divisible by $p$, so that hypothesis (1) of Theorem \ref{nonv1} is satisfied. Since $p$ divides  $\lambda_3 +1 =p$ and $\lambda_4 +1=p$, hypothesis (2) is also satisfied. Hence the map $\psi$ of Definition \ref{defpsi} induces a nonzero element of $\Hom_G(\Delta(\lambda), \Delta(\mu))$. \end{example}
	\begin{remark}
	Suppose in Theorem \ref{nonv1} that $\mu$ is a partition of length at most 2. In this case the hypothesis $\lambda \in P(\mu)$ is equivalent to $\mu_2 \le \la_1$. Hence we recover \cite[Theorem 3.1]{MS3}. In fact this result was the main motivation for Theorem \ref{nonv1} and the consideration of pairs of partitions $\lambda, \mu$ satisfying $\lambda \in P(\mu)$ in the present paper.
	\end{remark}
	
\subsection{Proof of Theorem \ref{nonv1}}\label{s5.2}
From the presentation of the Weyl module $\Delta(\lambda)$ given in (\ref{pres1}), it follows that in order to show that the map $\psi : D(\lambda) \to \Delta(\mu)$ of Definition \ref{defpsi} induces a map $\Delta(\lambda) \to \Delta(\mu)$ of Weyl modules, it suffices to prove the following lemma.
	\begin{lemma}\label{nonvlemma}
		Under the assumptions of Theorem \ref{nonv1}, the composition \[ D(\la(s,t)) \xrightarrow{\square_{\lambda,s,t,}}D(\lambda) \xrightarrow{\psi} \Delta(\mu) \]
		is the zero map for all $s \in \{1, 2, \dots, n-1\}$ and $i \in \{1, 2, \dots, \lambda_{s+1}\}$, where the map $\square_{\lambda,s,t,}$ is defined in (\ref{boxmap}).
	\end{lemma}
	\begin{proof}The proof of this lemma is somewhat lengthy and takes up most of Section \ref{s5.2}. Fix $s \in \{1, 2, \dots, n-1\}$. We distinguish 3 cases that depend on the relative sizes of $s$ and the length $m$ of $\mu$. 
	
	In Section 5.3 below, we have gathered four elementary lemmas needed in the proof of Lemma \ref{nonvlemma}.
	
	\textbf{Case 1.} Suppose $m \le s \le n-1$. From Definition \ref{defpsi} and Corollary \ref{triang} we have $\psi = \sum_{A} \phi_{T_{A}}$, where the sum ranges over all $A \in T_{n}(\mathbb{N})(\la, \mu)$. Since each matrix $A$ is upper triangular, Lemma \ref{comp2} yields \begin{equation}\label{511}
		\psi(e^{\la(s,t)})=\sum_{A}\sum_{t_1,\dots,t_{m}}\tbinom{a_{1s}+t_1}{t_1}\tbinom{a_{2s}+t_2}{t_2} \cdots \tbinom{a_{ms}+t_m}{t_m}[T_A(s,t;t_1,\dots,t_{m},0,\dots,0)],\end{equation}
where the left sum is over all $A=(a_{ij}) \in T_{n}(\mathbb{N})(\la, \mu)$ and the right sum is over all  $t_1,\dots,t_{m} \in \mathbb{N}$ such that $t_{1}+\dots+t_{m}=t$, \; $a_{1,s+1}-t_1 \ge 0, \; \dots, \; a_{m,s+1}-t_{m} \ge 0$.

Each tableau $T=T_A(s,t;t_1,\dots,t_{m},0,\dots,0)$, where
\begin{equation}\label{514}T=\begin{matrix*}[l]
		1^{(a_{11})} \cdots s^{(a_{1s}+t_1)}(s+1)^{(a_{1,s+1}-t_1)} \cdots n^{(a_{1n})} \\
		2^{(a_{22})} \cdots s^{(a_{2s}+t_2)}(s+1)^{(a_{2,s+1}-t_2)} \cdots n^{(a_{2n})} \\
		\cdots \\ 
		m^{(a_{mm})}\cdots s^{(a_{ms}+t_m)}(s+1)^{(a_{m,s+1}-t_m)} \cdots n^{(a_{mn})}   \end{matrix*}\end{equation}
in the right hand side of (\ref{511}) is semistandard by Lemma \ref{lem12}(1). Fix such a $T$. Its coefficient in the right hand side of (\ref{511}) is equal to   \begin{equation}\label{5115}\sum_{B, u_1,\dots,u_{m}}\tbinom{a_{1s}+t_1}{u_1}\tbinom{a_{2s}+t_2}{u_2} \cdots \tbinom{a_{ms}+t_m}{u_m}
	\end{equation} where the sum is over all $B=(b_{ij}) \in T_{n}(\mathbb{N})(\la, \mu)$ and all  $u_1,\dots,u_m \in \mathbb{N}$ such that for all $i=1,\dots,m$ 
	\begin{equation}
		\label{5116}b_{ij}=a_{ij}, (j\neq s, s+1), \; b_{is}+u_i=a_{is}+t_i, \; b_{i,s+1}-u_i=a_{i,s+1}-t_i, \; u_1+ \dots +u_m=t.
	\end{equation}
From Lemma \ref{lemma1} below, it follows that (\ref{5115}) is equal to \[\sum_{ u_1+\dots+u_{m}=t}\tbinom{a_{1s}+t_1}{u_1}\tbinom{a_{2s}+t_2}{u_2} \cdots \tbinom{a_{ms}+t_m}{u_m},\] where the sum is over all  $u_1,\dots,u_m \in \mathbb{N}$ such that $u_1+\dots+u_m=t$ and $a_{is}+t_i \ge u_i$ for all $i=1,\dots,m$. By Lemma \ref{bin2}(1) this is equal to \[\tbinom{a_{1s}+t_1 + \dots +a_{ms}+t_m}{t} = \tbinom{\la_s+t}{t},\] which by Lemma \ref{bin1}(2) is zero in $K$ by assumption (2) of the Theorem \ref{nonv1}.

\textbf{Case 2.} Suppose $1<s<m$. This case is more involved. Let $A=(a_{ij}) \in T_{n}(\mathbb{N})(\la, \mu)$, let $t$ be an integer such that $1 \le t \le \la_{s+1}$ and consider the tableau $T_A \in \mathrm{SST}_{\la}(\mu)$. As in the first paragraph of the proof of Case 2 of Theorem \ref{thm13} (the only difference is that here we use Remark \ref{remnewcond} in place of Corollary \ref{corhom}) we obtain that $\phi_{T_{A}}(e^{\la(s,t)})$ is equal to 
\begin{align}\label{521}
	&\sum_{t_1,\dots,t_{s}}\prod_{i=1}^{s}\tbinom{a_{is}+t_i}{t_i}(-1)^{t-\tau_s} \sum_{k_{s+1},\dots,k_n}\tbinom{a_{s+1,s+1}-(t-\tau_s)+k_{s+1}}{k_{s+1}}\prod_{j=s+2}^{n}\tbinom{a_{s+1,j}+k_{j}}{k_{j}}\\\nonumber&[T(s,t;t_1,\dots,t-\tau_s,0,\dots,0)_{k_{s+1},\dots,k_n}],\end{align} 
where  $T(s,t;t_1,\dots,t-\tau_s,0,\dots,0)_{k_{s+1},\dots,k_n}$ is given by (\ref{eq1110}), \[\tau_s=t_1 + \dots + t_{s-1},\] the left sum is over all  $t_1,\dots,t_{s} \in \mathbb{N}$ subject to (\ref{newcond}) and $a_{ss}+t-\tau_{s-1} \le \mu_s$,
and the right sum is over all  $k_{s+1},\dots,k_{n} \in \mathbb{N}$ subject to
\begin{enumerate} \item[(Sa)] $k_{s+1}+\dots+k_n=t-\tau_s$, \item[(Sb)] $k_{s+1} \le a_{s,s+1}-t_s, \ k_{s+2} \le a_{s,s+2}, \ \dots, \ k_n \le a_{sn}.$\end{enumerate}

For clarity, we break up the proof into 3 steps.

\textit{Step 1.} Before we compute $\psi = \sum_{A \in T_{n}(\mathbb{N})(\la, \mu)}\phi_{T_A}$, let us simplify (\ref{521}); we want to compute the sum with respect to $t_s$.

For the sake of conciseness, define \[k=k_{s+2}+\dots+k_n.\] Then from (Sa) we have
	 $k=(t-\tau_{s})-k_{s+1}$.
Now we substitute in (\ref{521}) and rearrange the left sum to obtain
\begin{align*}
	&\sum_{t_1,\dots,t_{s-1}}\prod_{i=1}^{s-1}\tbinom{a_{is}+t_i}{t_i} \sum_{t_s}(-1)^{t-\tau_{s}}\tbinom{a_{ss}+t_s}{t_s} \sum_{k_{s+1},\dots,k_n}\tbinom{a_{s+1,s+1}-k}{t-\tau_s -k}\\\nonumber&\prod_{j=s+2}^{n}\tbinom{a_{s+1,j}+k_{j}}{k_{j}} [T(s,t;t_1,\dots,t-\tau_s,0,\dots,0)_{k_{s+1},\dots,k_n}]. \end{align*}
Note that in the above expression, we may drop $k_{s+1}$ from the right sum since $k_{s+1}=t-\tau_s -k$. Hence we may pass the $t_s$ from the middle sum to the right to get \begin{align}\label{522}
	&\sum_{t_1,\dots,t_{s-1}}\prod_{i=1}^{s-1}\tbinom{a_{is}+t_i}{t_i} \sum_{k_{s+2},\dots,k_n}\sum_{t_s}(-1)^{t-\tau_{s}}\tbinom{a_{s+1,s+1}-k}{t-\tau_s -k}\tbinom{a_{ss}+t_s}{t_s}\\\nonumber&\prod_{j=s+2}^{n}\tbinom{a_{s+1,j}+k_{j}}{k_{j}} [T(s,t;t_1,\dots,t-\tau_s,0,\dots,0)_{k_{s+1},\dots,k_n}], \end{align}
	where the left sum is over all  $t_1,\dots,t_{s-1} \in \mathbb{N}$ subject to \begin{enumerate} \item[(S$'$1a)] $t-\tau_{i-1}-\sum_{u=i+1}^{s+1}a_{u,s+1} \le t_i\le \min\{a_{i,s+1},t-\tau_{i-1}\}, \ i \le s-1$, \item[(S$'$1b)] $a_{ss}+t-\tau_{s-1} \le \mu_s$,\end{enumerate}
	the middle sum is over all  $k_{s+2},\dots,k_{n} \in \mathbb{N}$ subject to
	\begin{enumerate} \item[(S$'$2a)] $k_{i} \le a_{si}, \  i \ge s+2$, \item[(S$'$2b)] $t-\tau_{s-1}-a_{s,s+1} \le k \le a_{s+1s+1}$\end{enumerate} and the right sum is over all  $t_s\in \mathbb{N}$ subject to \begin{enumerate}\item[(S$'$3)] $t-\tau_{s-1}-a_{s+1,s+1} \le t_s \le \min\{a_{s,s+1}, t-\tau_{s-1}\}.$ \end{enumerate}
	Indeed, it is straightforward to verify that conditions (Sa) and (Sb) are equivalent to (S${'}$2a), $t-\tau_{s-1}-a_{s,s+1} \le k \le t - \tau_s$ and $k_{s+1}=t-\tau_s-k$. Furthermore, $t-\tau_s \le a_{s+1,s+1}$ and since $\tbinom{a_{s+1,s+1}-k}{t-\tau_s - k}=0$ for all $k$ such that $t-\tau_s < k \le a_{s+1,s+1}$, we obtain (S$'$2a) and (S$'$2b).
	
	The rows $s$ and $s+1$ of the tableau $T(s,t;t_1,\dots,t-\tau_s,0,\dots,0)_{k_{s+1},\dots,k_n}$ in the right hand side of (\ref{522}) are
	\[	\begin{matrix*}[l]
		s^{(a_{ss}+t-\tau_{s-1})}(s+1)^{(a_{s,s+1}-(t-\tau_{s-1}-k))} (s+2)^{(a_{s,s+2}-k_{s+2})}\cdots n^{(a_{sn}-k_n)} \\
		(s+1)^{(a_{s+1,s+1}-k)} (s+2)^{(a_{s+1,s+2}+k_{s+2})}\cdots n^{(a_{s+1,n}+k_n)}
		  \end{matrix*}\]
	and thus $T(s,t;t_1,\dots,t-\tau_s,0,\dots,0)_{k_{s+1},\dots,k_n}$ does not depend on $t_s$. Hence we may write (\ref{522}) as  follows
	\begin{align}\label{523}
		&\sum_{t_1,\dots,t_{s-1}}\prod_{i=1}^{s-1}\tbinom{a_{is}+t_i}{t_i}  \sum_{k_{s+2},\dots,k_n} \  \prod_{j=s+2}^{n}\tbinom{a_{s+1,j}+k_{j}}{k_{j}} \\\nonumber&\bigg(\sum_{t_s} (-1)^{t-\tau_{s}}\tbinom{a_{s+1,s+1}-k}{t-\tau_s -k}\tbinom{a_{ss}+t_s}{t_s} \bigg) [T(s,t;t_1,\dots,t-\tau_s,0,\dots,0)_{k_{s+1},\dots,k_n}]. \end{align}
		Now we claim that in the right sum of (\ref{523}) we may assume that $t_s \le t-\tau_{s-1}-k$. Indeed, let \[c=t-\tau_{s-1}-k.\] If $t_s > c$, then $t-\tau_s-k = t-\tau_{s-1} -t_s-k =c-t_s<0$ and thus $\tbinom{a_{s+1,s+1}-k}{t-\tau_s -k}=0.$ Next, from the first inequality of (S$'$2b) we have $c\le a_{s,s+1}$. Hence $c \le \min\{a_{ss+1}, t-\tau_{s-1}\}$. Thus we conclude from (S$'$3) that in the right sum of (\ref{523}), $t_s$ ranges from 0 to $c$. Note that $a_{s+1, s+1}-k \le  a_{s+1, s+1} \le \mu_{s+1} \le a_{ss}$, where the last inequality comes from Lemma \ref{lem12}(1).  Hence we may apply the first equality of Lemma \ref{bin2}(2) to conclude that the right sum is equal to \[(-1)^k \tbinom{a_{ss}-a_{s+1,s+1}+t-\tau_{s-1}}{t-\tau_{s-1}-k}.\] Thus far, we have shown that (\ref{521}) is equal to 
	\begin{align}\label{525}
		&\sum_{t_1,\dots,t_{s-1}}\prod_{i=1}^{s-1}\tbinom{a_{is}+t_i}{t_i}  \sum_{k_{s+2},\dots,k_n}(-1)^k \tbinom{a_{ss}-a_{s+1,s+1}+t-\tau_{s-1}}{t-\tau_{s-1}-k}\prod_{j=s+2}^{n}\tbinom{a_{s+1,j}+k_{j}}{k_{j}} \\\nonumber& [T(s,t;t_1,\dots,t-\tau_s,0,\dots,0)_{k_{s+1},\dots,k_n}], \end{align}
	where the left sum is subject to (S$'$1), the right sum is subject to (S$'$2) and \begin{align}\label{526} &T(s,t;t_1,\dots,t-\tau_s,0,\dots,0)_{k_{s+1},\dots,k_n} = \\\nonumber &
		\begin{matrix*}[l]
		1^{(a_{11})} \cdots s^{(a_{1s}+t_{1})}(s+1)^{(a_{1,s+1}-t_1)} \cdots n^{(a_{1n})} \\
		\vdots \\ 
		(s-1)^{(a_{s-1,s-1})}s^{(a_{s-1,s}+t_{s-1})}(s+1)^{(a_{s-1,s+1}-t_{s-1})} \cdots n^{(a_{s-1,n})} \\
		s^{(a_{ss}+t-\tau_{s-1})}(s+1)^{(a_{s,s+1}-(t-\tau_{s-1}-k))} (s+2)^{(a_{s,s+2}-k_{s+2})}\cdots n^{(a_{sn}-k_n)} \\
		(s+1)^{(a_{s+1,s+1}-k)} (s+2)^{(a_{s+1,s+2}+k_{s+2})}\cdots n^{(a_{s+1,n}+k_n)}\\
		(s+2)^{(a_{s+2,s+2})}\cdots n^{(a_{s+2,n})} \\
		\vdots\\
		m^{(a_{mm})} \cdots n^{(a_{mn})}.  \end{matrix*}
\end{align}

\textit{Step 2}. Next,  we start the computation of $\psi(e^{\lambda(s,t)})$. As in the previous step, we will simplify the sums using change of variable arguments and Lemma \ref{bin2}.

 We have seen that $\psi = \sum_{A}\phi_{T_A}$, where the sum is over all $A=(a_{ij}) \in T_{n}(\mathbb{N})(\la, \mu)$. \\
 Claim. We may write the sum $\sum_{A}\phi_{T_A}$ as follows (isolating the $s+1$ row of $A$ and `forgetting' the $s$ row), \[\psi = \sum_{a_{ij}: i \notin \{s, s+1\}} \  \   \sum_{a_{s+1,s+1},\dots,a_{s+1,n}}\phi_{T_A},\] where the left sum is over all  $a_{ij} \in \mathbb{N}$, $1 \le i \le j \le n, i\neq s,s+1$, such that
\begin{enumerate}
	\item[(S$'$4a)] $\sum_{j=i}^{n} a_{ij}=\mu_i, \  i \neq s,s+1$, \  $\sum_{i=1}^{j} a_{ij}=\lambda_j, \  j \le s-1$,
	\item[(S$'$4b)] $ \sum_{i=1, i \neq s,s+1}^{j}a_{ij} \le \lambda_j, \  j \ge s$,
\end{enumerate}	
and the right sum is over all  $a_{s+1,s+1},\dots,a_{s+1,n} \in \mathbb{N}$, such that
\begin{enumerate}
	\item[(S$'$5a)] $\sum_{j=s+1}^{n} a_{s+1,j}=\mu_{s+1}$,
	\item[(S$'$5b)] $ a_{s+1,j} \le \lambda _j - \sum_{i=1, i \neq s,s+1}^{j}a_{ij}, \  j \ge s+1$.
\end{enumerate}	
Proof of Claim. First, suppose $A=(a_{ij}) \in T_{n}(\mathbb{N})(\la, \mu)$. Since $A$ is upper triangular with sequence of column sums equal to $\lambda$, we obtain the second equality of (S$'$4a), (S$'$4b) and (S$'$5b). Similarly, since  $A$ is upper triangular with sequence of row sums equal to $\mu$, we obtain the first equality of (S$'$4) and (S$'$5a). Conversely, suppose \begin{itemize} \item we have $a_{ij} \in \mathbb{N}$, $1 \le i \le j \le n, i\neq s,s+1$ satisfying (S$'$4), and \item we have $a_{s+1,s+1},\dots,a_{s+1,n} \in \mathbb{N}$ satisfying  (S$'$5).\end{itemize} We define  \begin{equation}\label{ss}a_{sj}=\lambda_j- \sum_{i=1, i \neq s}^{j}a_{ij}, \  j=s,\dots,n, \end{equation} and \[a_{ij}=0, i>j.\]Then the $n \times n$ matrix $(a_{ij})$ is in $T_n(\mathbb{N})(\lambda, \mu)$. Furthermore, the $i$ row of $(a_{ij})$, for $i \neq s,s+1$ is $(0 \  \cdots 0 \   a_{ii}\  \cdots a_{in})$ and the $s+1$ row of $(a_{ij})$ is $0 \  \cdots 0 \  a_{ss} \  \cdots a_{s+1,n}.$ Finally, $(a_{ij})$ is the unique matrix in $T_n(\mathbb{N})(\lambda, \mu)$ with these rows in position $1, \dots s-1, s+1 \dots n$. The proof of the Claim  is complete.

Now using (\ref{525}) and swapping the sums $\sum_{a_{s+1,s+1},\dots,a_{s+1,n}}$ and  $\sum_{t_1,\dots,t_{s-1}}$ (which is permissible since, from (S$'$1) and (\ref{ss}) (for $j=s+1$), we can see that $t_1,\dots,t_{s-1}$ are independent of $a_{s+1,s+1},\dots,a_{s+1,n}$), we obtain
	\begin{align}\label{600}
	\psi(e^{\lambda(s,t)})= 	&\sum_{a_{ij}: i \notin \{s, s+1\}} \  \sum_{t_1,\dots,t_{s-1}}\prod_{i=1}^{s-1}\tbinom{a_{is}+t_i}{t_i} \\ \nonumber & \sum_{\substack {a_{s+1,s+1},\dots,a_{s+1,n}, \\ k_{s+2},\dots,k_n }  } \   (-1)^k \tbinom{a_{ss}-a_{s+1,s+1}+t-\tau_{s-1}}{t-\tau_{s-1}-k}\prod_{j=s+2}^{n}\tbinom{a_{s+1,j}+k_{j}}{k_{j}} \\\nonumber& [T(s,t;t_1,\dots,t-\tau_s,0,\dots,0)_{k_{s+1},\dots,k_n}], \end{align}
	where the right sum is subject to (S$'$2) and (S$'$5). 
	
	Before we make substitutions after changing variables, let us consider the right sum in (\ref{600}). Let $I$ be the set of all sequences \[(a_{s+1,s+1},\dots,a_{s+1,n},k_{s+2},\dots,k_{n})\] of nonnegative integers satisfying (S$'$2) and (S$'$5). If $q_{s+2},..,q_n \in \mathbb{N}$ satisfy \begin{equation}\label{qin}q_j \le \lambda_j - \sum_{i=1, i \neq s,s+1}^{j}a_{ij}, \  j \ge s+2,\end{equation}
	define the set $I_{q_{s+2},\dots,q_n}$ consisting of all \[(a_{s+1,s+1},\dots,a_{s+1,n},k_{s+2},\dots,k_{n}) \in I\] such that \begin{equation}\label{qsum}a_{s+1,j}+k_j=q_j, j \ge s+2.\end{equation} Then we have the disjoint union \[I= \cup_{q_{s+2},\dots,q_n}I_{q_{s+2},\dots,q_n}.\] From the definitions we have that, given $q_{s+2},\dots,q_n$ as in (\ref{qin}), then\\ $(a_{s+1,s+1},\dots,a_{s+1,n},k_{s+2},\dots,k_{n}) \in I_{q_{s+2},\dots,q_n}$ if and only if (S$'$2), (S$'$5) and (\ref{qsum}) hold. Hence we may apply Lemma \ref{lemma2} below to conclude the following. Suppose $q_{s+2},\dots,q_n$ satisfy (\ref{qin}). Then 
		\[(a_{s+1,s+1},\dots,a_{s+1,n},k_{s+2},\dots,k_{n}) \in I_{q_{s+2},\dots,q_n}\] if and only if (S$'$5b), (S$'$6) and (\ref{qsum}) hold, where
	\begin{enumerate}
		\item[(S$'$6a)]$a_{s+1,s+1}=\mu_{s+1}-q+k$,
		\item[(S$'$6b)]$t-\tau_{s-1}-\lambda_{s+1}+\sum_{i=1}^{s-1}a_{i,s+1}+\mu_{s+1} \le q \le \mu_{s+1}$,
	\end{enumerate}
	and $q=q_{s+2}+\dots+q_n$.
	
	We also note that $I_{q_{s+2},\dots,q_n}$ is nonempty as it contains $(\mu_{s+1}-q,q_{s+2}, \dots, q_n, 0, \dots,0)$.
	
		Now with the substitutions \[q_j=a_{s+1,j}+k_j, \  (j \ge s+2),\] we see that rows $s$ and $s+1$ of the tableau $T(s,t;t_1,\dots,t-\tau_s,0,\dots,0)_{k_{s+1},\dots,k_n}$ in the right hand side of (\ref{600}), which is given in (\ref{526}), are
	\begin{align}\label{ss+1}
		&s^{(\lambda_s -\sum_{i=1}^{s-1}a_{is}+t-\tau_{s-1})}        (s+1)^{(\lambda_{s+1}-\mu_{s+1}-\sum_{i=1}^{s-1}a_{i,s+1}-(t-\tau_{s-1})+q)}  \\ \nonumber    
		&\;\;\;\;\;\;\;\;\;(s+2)^{(\lambda_{s+2}-\sum_{i=1, i \neq s,s+1}^{n}a_{i,s+2}-q_{s+2})}\cdots n^{(\lambda_{n}-\sum_{i=1, i \neq s,s+1}^{n}a_{in}-q_{n})} \\ \nonumber
		&(s+1)^{(\mu_{s+1}-q)}   (s+2)^{(q_{s+2})}\cdots  n^{(q_{n})}.
	\end{align} From (S$'$6a) it follows that $\tbinom{a_{ss}-a_{s+1,s+1}+t-\tau_{s-1}}{t-\tau_{s-1}-k}=\tbinom{a_{ss}-\mu_{s+1}+q-k+t-\tau_{s-1}}{t-\tau_{s-1}-k}.$ The point is we have expressed (\ref{526}) independently of the $a_{s+1,s+1},\dots,a_{s+1,n},k_{s+2},\dots,$ $k_n$. Using this and the conclusion of the previous paragraph,  we may rewrite the right sum in (\ref{600}) as follows, 
		\begin{align}\label{611}
		& \sum_{q_{s+2},\dots,q_n} \  \bigg(\sum_{\substack {a_{s+1,s+1},\dots,a_{s+,1n}, \\ k_{s+2},\dots,k_n }  }  (-1)^k \tbinom{a_{ss}-\mu_{s+1}+q-k+t-\tau_{s-1}}{t-\tau_{s-1}-k}\prod_{j=s+2}^{n}\tbinom{q_{j}}{k_{j}}\bigg) \\\nonumber& [T(s,t;t_1,\dots,t-\tau_s,0,\dots,0)_{k_{s+1},\dots,k_n}], \end{align}
	where the left sum is over all $q_{s+2},\dots,q_n \in \mathbb{N}$ such that (S$'$6b) and (\ref{qin}) hold, and the right sum is over all $a_{s+1,s+1},\dots,a_{s+1,n},k_{s+2},\dots,k_n \in \mathbb{N}$ such that (S$'$5b), (S$'$6a) and (\ref{qsum}) hold. Then by Lemma \ref{lemma3} below, the right sum in (\ref{611}) is equal to 
	\begin{equation}\label{sum1}\sum_{ k_{s+2},\dots,k_n }  (-1)^k \tbinom{a_{ss}-\mu_{s+1}+q-k+t-\tau_{s-1}}{t-\tau_{s-1}-k}\prod_{j=s+2}^{n}\tbinom{q_{j}}{k_{j}}\end{equation}
where the sum is over all $k_{s+2},\dots,k_n \in \mathbb{N}$ such that 
$k_j \le q_j  \  (j\ge s+2)$ and $k \le \lambda_{s+1}-\sum_{i=1}^{s-1}a_{i,s+1}-\mu_{s+1}+q.$ We observe that from (S$'$1a) and (S$'$6b) we have \[ 0 \le t-\tau_{s-1} \le \lambda_{s+1}-\sum_{i=1}^{s-1}a_{i,s+1}-\mu_{s+1} + q.\] Moreover, by our convention $\tbinom{a_{ss}-\mu_{s+1}+q-k+t-\tau_{s-1}}{t-\tau_{s-1}-k}=0$ for all $k$ such that 
$t-\tau_{s-1} <k$. Thus the sum in (\ref{sum1}), say $\Sigma$, is over all $k_{s+2},..,k_n \in \mathbb{N}$ such that $k_j \le q_j  \  (j \ge s+2)$ and $k \le t-\tau_{s-1}.$ Therefore \begin{align}
	\Sigma &= \sum_{k=0}^{t-\tau_{s-1}}(-1)^k\tbinom{a_{ss}-\mu_{s+1}+q-k+t-\tau_{s-1}}{t-\tau_{s-1}-k}\sum_{ k_{s+2}+\dots+k_n=k } \  \prod_{j=s+2}^{n}\tbinom{q_{j}}{k_{j}} \\\nonumber &=\sum_{k=0}^{t-\tau_{s-1}}(-1)^k\tbinom{a_{ss}-\mu_{s+1}+q-k+t-\tau_{s-1}}{t-\tau_{s-1}-k}\tbinom{q}{k}
\end{align}	
by Lemma \ref{bin2}(1). From Lemma \ref{lem12}(1) we have $a_{ss}-\mu_{s+1} \ge 0$, so by letting $a=a_{ss}-\mu_{s+1}+q$ and $b=q$, we have $a \ge b \ge 0$. From the second identity of Lemma \ref{bin2}(2) we conclude that $
\Sigma = \tbinom{a_{ss}-\m_{s+1}+t-\tau_{s-1}}{t-\tau_{s-1}}$. 

To summarize, we have
\begin{align}\label{620}
	\psi(e^{\lambda(s,t)})= 	&\sum_{a_{ij}: i \notin \{s, s+1\}} \  \sum_{t_1,\dots,t_{s-1}} \  \sum_{q_{s+2},\dots,q_{n}}     \tbinom{a_{ss}-\m_{s+1}+t-\tau_{s-1}}{t-\tau_{s-1}}\prod_{i=1}^{s-1}\tbinom{a_{is}+t_i}{t_i} \\\nonumber& [T(s,t;t_1,\dots,t-\tau_s,0,\dots,0)_{(q_{s+2},\dots,q_n)}], \end{align}	
where the left sum is over (S$'$4), the middle sum is over (S$'$1a) and (S$'$1b) with $a_{ss}$ replaced by $\lambda_{s}-\sum_{i=1, i\neq s}^{n}a_{is}$, and the right sum is over (S$'$6b) and (\ref{qin}). Moreover, all rows of $T(s,t;t_1,\dots,t-\tau_s,0,\dots,0)_{(q_{s+2},\dots,q_n)}$, except rows $s$ and $s+1$, are equal to the corresponding rows of (\ref{526}), while rows $s$ and $s+1$ are given by (\ref{ss+1}).
		
\textit{Step 3}. We conclude the computation of $\psi(e^{\lambda(s,t)})$ with another change of variable argument.

Let us begin by rewriting the first sum in (\ref{620}) by isolating columns $s, s+1$. We see that $\psi(e^{\lambda (s,t)})$ is equal to 
\begin{align}\label{621}
	&\sum_{a_{ij}: i,j \notin \{s, s+1\}} \  \sum_{\substack{t_1,\dots,t_{s-1}\\a_{1s},\dots,a_{s-1,s+1}}} \  \sum_{q_{s+2},\dots,q_{n}}     \tbinom{a_{ss}-\m_{s+1}+t-\tau_{s-1}}{t-\tau_{s-1}}\prod_{i=1}^{s-1}\tbinom{a_{is}+t_i}{t_i} \\\nonumber& [T(s,t;t_1,\dots,t -\tau_{s},0,\dots,0)_{(q_{s+2},\dots,q_n)}], \end{align}
 where the left sum is over all $a_{ij} \in \mathbb{N}$, where $1 \le i \le j \le n$ and $i,j \notin \{s,s+1\}$, such that
\begin{enumerate}
	\item[(S$''$1a)] $\sum_{j=i, j \neq s, s+1}^{n} a_{ij} \le\mu_i, \  i \le s-1,$ and $\sum_{j=i}^{n} a_{ij}=\mu_i, \  i\ge s+2$,
	\item[(S$''$1b)]  $\sum_{i=1}^{j} a_{ij} = \lambda_j, \  j \le s-1,$ and $\sum_{i=1, i \neq s, s+1}^{j} a_{ij} \le \lambda_j, \  j \ge s+2,$
\end{enumerate}	
and the middle sum is over all  $t_1,\dots,t_{s-1} \in \mathbb{N} $ and $a_{1s},a_{1,s+1},\dots,a_{s-1,s} , a_{s-1,s+1} \in \mathbb{N}$, such that (S$'$1a) and (S$'$1b) hold and 
\begin{enumerate}
	\item[(S$''$2a)] $\sum_{j=i}^{n} a_{ij} = \mu_i, i \le s-1,$
	\item[(S$''$2b)]  $\sum_{i=1}^{s-1} a_{is} \le \lambda_s$ and $\sum_{i=1}^{s-1} a_{i,s+1} \le \lambda_{s+1}.$
\end{enumerate}	
Our goal in Step 3 is to compute the coefficient of $[T(s,t;t_1,\dots,t -\tau_{s},0,\dots,$ $0)_{(q_{s+2},\dots,q_n)}]$. Define 	\begin{center}
	$u_i=a_{is}+t_i$ ($i=1,\dots,s-1$) and $u=u_1+\dots+u_{s-1}$.
\end{center} Then it is easy to verify by direct substitution that \[T(s,t;t_1,\dots,t-\tau_s,0,\dots,0)_{(q_{s+2},\dots,q_n)}=T_D,\] where $D=(d_{ij}) \in T_{n}(\mathbb{N})(\la, \mu)$ is defined by 	\begin{align}	
	d_{is}&=u_i, \  d_{is+1}=\mu_i-\sum_{j=i, j\neq s,s+1}^na_{ij}-u_i, \  i \le s-1, \\\nonumber
	d_{ss}&=\lambda_s-t+u, \\\nonumber d_{s,s+1}&=\lambda_{s+1} -\mu_{s+1}-t+q-\sum_{i=1}^{s-1}\mu_i+\sum_{i=1}^{s-1} \  \sum_{j=i, j \neq s,s+1}^{n} a_{ij}+u,  \\\nonumber
	d_{sj}&=\lambda_j-\sum_{i=1, i\neq s,s+1}^{j} a_{ij}-q_j, \ j \ge s+2,\\\nonumber
	d_{s+1,s+1}&=\mu_{s+1}-q, \  d_{s+1,j}=q_j, \  j \ge s+1,\\\nonumber
d_{ij}&=a_{ij}, \  \text{otherwise}.
	\end{align}		
We have expressed $T_D$ independently of the $t_1,\dots,t_{s-1}$ and $a_{1s},a_{1,s+1},\dots,a_{s-1,s},$ $ a_{s-1,s+1}$. Moreover, we have \[u=\sum_{i=1}^{s-1}a_{is} +\sum_{i=1}^{s-1}t_i=\sum_{i=1}^{s-1}a_{is} +t-\tau_{s-1}\] and hence from (\ref{ss}) we conclude that $\tbinom{a_{ss}-\m_{s+1}+t-\tau_{s-1}}{t-\tau_{s-1}}=\tbinom{\lambda_s -\mu_{s+1}+t-u}{t-\tau_{s-1}}$. We observe that the left hand side of (S$'$6b) is equal to \[\mu_{s+1}-\lambda_{s+1}+t+\sum_{i=1}^{s-1}\mu_i -\sum_{i=1}^{s-1}\  \sum_{j=i, j \neq s, s+1}^{n}a_{ij} -u,\] that is, it is independent of the subscript $s+1$ of $a_{i,s+1}$. Hence from (S$'$1a), (S$''$2) and (S$'$6) we may swap the middle and right sums in (\ref{621}). These remarks allow us to obtain
\begin{equation}\label{631}
	\sum_{a_{ij}: i,j \notin \{s, s+1\}} \ \sum_{u_1,\dots,u_{s-1}} \  \sum_{q_{s+2},\dots,q_{n}} \bigg( \sum_{\substack{t_1,\dots,t_{s-1}\\a_{1s},\dots,a_{s-1,s+1}}}      \tbinom{\lambda_s -\mu_{s+1}+t-u}{t-\tau_{s-1}}\prod_{i=1}^{s-1}\tbinom{u_i}{t_i} \bigg) [T_D], \end{equation}
where the second sum is over all $u_1,\dots,u_{s-1} \in \mathbb{N}$ such that, given $a_{ij} \  (i,j \notin \{s, s+1\})$ as in the first sum,  there exist $t_1,..,t_{s-1} \in \mathbb{N}$ and $ a_{1s}, a_{1,s+1},\dots,a_{s-1,s}, a_{s-1,s-1} \in \mathbb{N}$ satisfying (S$''$2), (S$'$1a), (S$'$1b) with $a_{ss}$ replaced by $\lambda_s- \sum_{i=1, i \neq s}^{n}a_{is}$, and
\begin{enumerate}
	\item[(S$''$3)] $u_i=a_{is}+t_i$, \  $i=1,\dots,s-1$.
\end{enumerate}
The fourth sum is over (S$''$2a), (S$''$2b), (S$'$1a) and (S$'$1b). However, condition (S$'$1b) is equivalent to $t \le \mu_s -\lambda_s +u$. Thus we may impose this condition on the second sum in (\ref{631}) and not on the fourth.

By Lemma \ref{lemma4} below,  (\ref{631}) may be written as follows
	\begin{align}\label{650}
		&\sum_{a_{ij}: i,j \notin \{s, s+1\}} \  \sum_{u_1,\dots,u_{s-1}} \  \sum_{q_{s+2},\dots,q_{n}} \bigg( \sum_{t_1,\dots,t_{s-1}}   \tbinom{\lambda_s -\mu_{s+1}+t-u}{t-\tau_{s-1}}\prod_{i=1}^{s-1}\tbinom{u_i}{t_i}\bigg) [T_D], \end{align}	
where the fourth sum is over all $t_1,\dots,t_{s-1} \in \mathbb{N}$	such that
\begin{enumerate} \item[(S$''$4)] $\tau_{s-1} \le t$,\  $0 \le u- \tau_{s-1} \le \lambda_{s}$ and $t_i \le u_i, \ i=1,\dots,s-1.$\end{enumerate}
We claim that 	$\lambda_s -\mu_{s+1}+t-u \ge 0$. Indeed, at the end of Step 2 we noticed that $a_{ss}-\m_{s+1}+t-\tau_{s-1} \ge 0$. We also noticed that under the change of variables in Step 3, $a_{ss}-\m_{s+1}+t-\tau_{s-1}=\lambda_s -\mu_{s+1}+t-u$, so the claim  follows. We also note that we may drop the second condition in (S$''$4) (this follows from the third condition of (S$''$4) and the definition of $u$).  Hence we apply Lemma \ref{bin2}(1) to conclude that (\ref{650}) is equal to 
\begin{align}\label{700}
	&\sum_{a_{ij}: i,j \notin \{s, s+1\}} \  \sum_{u_1,\dots,u_{s-1}} \  \sum_{q_{s+2},\dots,q_{n}}   \tbinom{\lambda_s -\mu_{s+1}+t}{t}[T_D]. \end{align}				
Now from (\ref{526}) we have that the number of appearances of the elements $1,\dots,s-1,s$ in $T_D$ is $\lambda_1+\dots+\lambda_{s-1} +(\lambda_{s}+t)$ and these appear in the first $s$ rows of $T_D$. Hence $\lambda_1+\dots+\lambda_{s-1} +(\lambda_{s}+t) \le \mu_1+\dots+\mu_s$, that is $t \le c_s$. On the other hand we have $t \le \lambda_{s+1}$ and thus $t \le \min\{c_s, \lambda_{s+1}\}$ for all $1<s<m$. From Lemma \ref{rem11}(1), $\min\{c_s, \lambda_{s+1}\}=c_s$ if $s<m-1$. Now from assumption (1) of the theorem and Lemma \ref{bin1}(2), it follows that (\ref{700}) is equal to $0$.
				
	\textbf{Case 3.} Suppose $s=1$.	This is essentially identical to the first part of the proof of \cite[Theorem 3.1]{MS3} and thus omitted. The only difference is that we append rows $3,\dots,m$ of $T_A$ to the two-rowed tableaux that appear in the proof of loc.cit. according to Lemma \ref{insertrows}.
	
	The proof of Lemma \ref{nonvlemma} will be complete once we prove the four lemmas of Section 5.3 below. \end{proof}
	
	From Lemma \ref{nonvlemma} we have that $\psi$ induces a $G$-homomorphism \[\bar{\psi} : \Delta(\lambda) \to \Delta(\mu).\] It remains to be shown that $\bar{\psi} \neq 0$. Recall that we are assuming $\lambda \unlhd \mu$ in the statement of Theorem \ref{nonv1}.
	
	\begin{lemma}\label{nonvlemma2}
	The map $\bar{\psi} : \Delta(\lambda) \to \Delta(\mu)$ is nonzero.
	\end{lemma}
	\begin{proof}From Remark \ref{remarkphi} we have that \[\bar{\psi}(1^{(\lambda_1)} \otimes \dots \otimes n^{(\lambda_n)})=\sum_{T \in \mathrm{SST}_{\la}(\mu)}[T],\] which is a sum of certain distinct basis elements of $\Delta{(\mu)}$ according to Theorem \ref{standard}. So it suffices to show that the set $\mathrm{SST}_{\la}(\mu)$ is nonempty. This is indeed the case since $\lambda \unlhd \mu$,  see \cite[Theorem 1'']{Lam}. \end{proof}
	
	The proof of Theorem \ref{nonv1} will be complete once we prove the next four lemmas.
	
	\subsection{Four elementary lemmas}\label{ellemmas}We prove here the four  lemmas that were used in the proof of Lemma \ref{nonvlemma}. 
	
	Recall that we have partitions $\lambda, \mu \in \Lambda^+(n,r)$ and integers $s,t$ satisfying $1 \le s < m$ and $1 \le t \le \lambda_{s+1}$, where $m=\ell(\mu)$. 

	\subsubsection{The first lemma}Let us recall the setup of Case 1 of the proof of Theorem \ref{nonv1}. Suppose $m \le s\le n-1$. We have a fixed tableau \begin{equation}\label{lem611}T=\begin{matrix*}[l]
			1^{(a_{11})} \cdots s^{(x_{1s})}(s+1)^{(y_{1,s+1})} \cdots n^{(a_{1n})} \\
			2^{(a_{22})} \cdots s^{(x_{2s})}(s+1)^{(y_{2,s+1})} \cdots n^{(a_{2n})} \\
			\cdots \\ 
			m^{(a_{mm})}\cdots s^{(x_{ms})}(s+1)^{(y_{m,s+1})} \cdots n^{(a_{mn})}   \end{matrix*}\end{equation} given by eq. (\ref{514}), where $(a_{ij}) \in T_n(\mathbb{N})(\la, \mu)$, $x_{is}=a_{is}+t_i$ and $y_{i,s+1}=a_{i,s+1}-t_i$. For the matrix $A_T$ of $T$ we have \begin{equation}\label{lem612}
		A_T \in T_n(\mathbb{N})(\la(s,t), \mu).
	\end{equation}
	
		Let $X$ be the set of all ordered $m+1$ tuples \[(B, u_1, \dots u_m),\] where $B=(b_{ij}) \in T_{n}(\mathbb{N})(\la, \mu)$ and  $u_1,\dots,u_m \in \mathbb{N}$ satisfy for all $i=1,\dots,m$ the following conditions
		\begin{equation}
			\label{lem613}b_{ij}=a_{ij}, (j\neq s, s+1), \  b_{is}+u_i=x_{is}, \  b_{i,s+1}-u_i=y_{i,s+1}, \  u_1+ \dots +u_m=t.
		\end{equation}
		
		Let $X'$ be the set of all ordered $m$ tuples \[(u_1, \dots, u_m),\] of nonnegative integers such that \begin{equation} \label{lem614} 
			x_{is} \ge u_i  \  (i=1, \dots, m), \  u_1 + \dots +u_m = t.
		\end{equation}
		\begin{lemma}\label{lemma1}
			The map $X \to X', (B,u_1,\dots,u_m) \to (u_1,\dots u_m),$ is a bijection.
		\end{lemma}
		\begin{proof}
			Suppose $(B,u_1,\dots,u_m) \in X$. From the second equality of (\ref{lem613}) we have $x_{is}-u_i=b_{is} \ge 0$. Thus (\ref{lem614}) is satisfied and $(u_1,\dots u_m) \in X'$.
			
			Suppose $(B,u_1,\dots,u_m),  (C,u_1,\dots,u_m)\in X$. From the first three equalities of (\ref{lem613}) we have $B=C$. Hence the map $X \to X'$ is 1-1.
			
			Suppose $(u_1,\dots,u_m) \in X'$. Define $B=(b_{ij})$ from the first three equalities of (\ref{lem613}). From the inequalities of (\ref{lem614}) it follows that the entries of $B$ are nonnegative integers. From (\ref{lem612}) and the equality $u_1+\dots+u_m=t$, if follows that $B=(b_{ij}) \in T_{n}(\mathbb{N})(\la, \mu)$. Hence the map $X \to X'$ is onto.
		\end{proof}
		\subsubsection{The second lemma} Suppose $1 < s < m$.
		\begin{lemma}\label{lemma2}Suppose we have nonnegative integers
		\begin{itemize}
			\item $a_{ij}$, where $1 \le i \le j \le n$, satisfying (\ref{ss}),
			\item $t_1, \dots, t_{s-1}$ satisfying (S$'$1),
			\item $q_{s+2}, \dots, q_n$ satisfying (\ref{qin}), and
			\item $k_{s+2}, \dots, k_n$ satisfying (\ref{qsum}).
		\end{itemize}
		Then the tuple $w=(a_{s+1,s+1}, \dots, a_{s+1,n},k_{s+2}, \dots, k_n)$ satisfies (S$'$2) and (S$'$5) if and only if it satisfies (S$'$5b) and (S$'$6).
	\end{lemma}
	\begin{proof} We recall the notation \[\tau_s=t_1 + \dots + t_{s-2},  \  q=q_{s+2}+\dots+q_n,                     \ k  = k_{s+2} + \dots + k_n.\]Let us begin with an equation that will be used in both directions of the proof of the lemma. By summing eqs. (\ref{qsum}) with respect to $j=s+2, \dots, n$, we obtain \begin{equation}\label{both}\sum_{j=s+2}^na_{s+1,j}+k=q.\end{equation}

Suppose $w$ satisfies (S$'$2) and (S$'$5).  Using (S$'$5a) and (\ref{both}) we get $\mu_{s+1}-a_{s+1,s+1}+k=q$. Hence  $a_{s+1,s+1}=\mu_{s+1}-q+k$ which is eq. (S$'$6a).

From (S$'$2b) we have \begin{align*}&t-\tau_{s-1}-a_{s,s+1} \le k \le a_{s+1s+1} \Rightarrow (\mathrm{add} \  q-k) \\ &
	t-\tau_{s-1}-a_{s,s+1}+q-k \le q \le a_{s+1,s+1}	+q-k \Rightarrow (\mathrm{use \  (S'6a)}) \\&
	t-\tau_{s-1}-a_{s,s+1}-a_{s+1,s+1}+\m_{s+1} \le q \le \mu_{s+1} \Rightarrow (\mathrm{use \  (S'5b) \  for} \  j=s+1) \\&
	t-\tau_{s-1}-a_{s,s+1}-\lambda_{s+1}+\sum_{i=1, i \neq s,s+1}^{s+1}a_{i,s+1} \le q \le \mu_{s+1} \Rightarrow \\&
	t-\tau_{s-1}-\lambda_{s+1}+\sum_{i=1}^{s-1}a_{i,s+1} \le q \le \mu_{s+1}.
\end{align*}
Hence we have shown (S$'$6b).

Conversely, suppose $w$ satisfies (S$'$5b) and (S$'$6). Then the implications in the proof of the previous paragraph may be reversed. Thus we have  (S$'$2b). 

Using (\ref{ss}) we have \begin{align*}
a_{sj}&=\lambda_j- \sum_{i=1, i \neq s}^{j}a_{ij} \\&=
\lambda_j- \sum_{i=1, i \neq s, s+1}^{j}a_{ij}-a_{s+1,j} \  \   (\mathrm{use \  (\ref{qsum})}) \\&=\lambda_j- \sum_{i=1, i \neq s, s+1}^{j}a_{ij}-q_{j} +k_j \  \   (\mathrm{use \  (\ref{qin})})\\&
\ge k_j,
\end{align*}
for  all $j \ge s+2$. Hence we have shown (S$'$2a).

Finally, from (S$'$6a) we have \begin{align*}a_{s+1,s+1}&=\mu_{s+1}-q+k \; \;  (\mathrm{use \  (\ref{both})})\\&
	=\mu_{s+1} -\sum_{j=s+2}^n a_{s+1,j}.	\end{align*}
Hence $\sum_{j=s+1}^{n} a_{s+1,j}=\mu_{s+1}$ and we have proven (S$'$5a). \end{proof}
		
	\subsubsection{The third lemma}
	Let us recall the setup of Case 2, Step 2 of the proof of Theorem \ref{nonv1}. Suppose $1 < s < m$. We have $a_{ij} \in \mathbb{N}$, where $1 \le i \le j \le n$ and $i,j \notin \{s, s+1\}$, that satisfy (S$'$4). We have $t_1,\dots,t_{s-1} \in \mathbb{N}$ that satisfy (S$'$1a). Also we have $q_{s+2},\dots,q_n \in \mathbb{N}$ satisfying (S$'$6b) and (\ref{qin}). 
	
	Let $B$ be the set of all sequences  \[(a_{s+1,s+1},\dots,a_{s+1,n},k_{s+2},\dots,k_{n})\] of nonnegative integers that satisfy (S$'$5b), (S$'$6a), (\ref{qsum}) and let $B'$ be the set of all sequences  \[(k_{s+2},\dots,k_{n})\] of nonnegative integers such that \begin{equation}\label{inlemma3} k_j \le q_j  \  (j=s+2,\dots,n) \  \text{and} \  k \le \lambda_{s+1}-\sum_{i=1}^{s-1}a_{i,s+1}-\mu_{s+1}+q.\end{equation}
	Recall the notation \begin{center}
		$k= k_{s+2}+\dots+k_n, \  q=q_{s+2}+\dots+q_n$ and $\tau_i=t_1+\dots+t_i$.
	\end{center}
	\begin{lemma}\label{lemma3}The map $f:B \to B'$, \[(a_{s+1,s+1},\dots,a_{s+1,n},k_{s+2},\dots,k_{n}) \mapsto (k_{s+2},\dots,k_{n}),\] is a bijection.
		\begin{proof} Let $x=(a_{s+1,s+1},\dots,a_{s+1,n},k_{s+2},\dots,k_{n}) \in B$. From (\ref{qsum}), $k_j=q_j-a_{s+1j} \le q_j$. Also, from (S$'$5b)  for $j=s+1$ and from (S$'$6a) we have $k+\mu_{s+1}-q \le \lambda_{s+1}-\sum_{i=1}^{s-1}a_{i,s+1}$ from which the second inequality of (\ref{inlemma3}) follows. Hence $f(B) \subseteq B'$. Clearly $f$ is injective. If $(k_{s+2},\dots,k_{n}) \in B'$, define $a_{s+1,s+1}$ from (S$'$6a) and define $a_{s+1,j}$ from (\ref{qsum}). Then $a_{s+1,s+1} \ge \m_{s+1}-q \ge 0$ by the second inequality in (S$'$6b). From the first inequality of (\ref{inlemma3}) we have $a_{s+1,j} \ge 0$ for all $j=s+2,\dots,n$. From (S$'$6a) and the second inequality of (\ref{inlemma3}) it follows that $a_{s+1,s+1} \le \lambda_{s+1} -\sum_{i=1, i \neq s,s+1}^{s+1}a_{i,s+1}$ and hence (S$'$5b) holds for $j=s+1$. From (\ref{qin}) we have $a_{s+1,j}=q_j-k_j \le \lambda_j-\sum_{i=1, i \neq s,s+1}^{j}a_{ij}$, and hence 
			(S$'$5b) holds for $j=s+2,\dots,n$. We have shown that (S$'$5b) holds. Note that (S$'$6a) and (\ref{qsum}) hold by definition. Hence we have established that $(a_{s+1,s+1},\dots,a_{s+1,n},$ $k_{s+2},\dots,k_n)$ $ \in B$. This proves that $f$ is surjective.
		\end{proof}
	\end{lemma}
	
		\subsubsection{The fourth lemma}Let us recall the setup and relevant notation of Case 2, Step3 of the proof of Theorem \ref{nonv1}. We have $a_{ij} \in \mathbb{N}$, where $\{i,j\} \notin \{s,s+1\}$, and we have $u_1,\dots,u_{s-1} \in \mathbb{N}$ as in (\ref{631}). Let $u=u_1+\dots+u_{s-1}$. Define the set $C$ consisting of all sequences of nonnegative integers \[(t_1,\dots,t_{s-1},a_{1s}, a_{1,s+1},\dots,a_{s-1,s}, a_{s-1,s+1})\] that satisfy (S$'$1a), (S$''$2) and (S$''$3).  Define the set $C'$ consisting of all sequences of nonnegative integers \[(t_1,\dots,t_{s-1})\] that satisfy (S$''$4). 
	
	For later use, if $(t_1,\dots,t_{s-1},a_{1s},$ $a_{1,s+1},\dots,a_{s-1,s}, a_{s-1,s+1}) \in C$, define \begin{center}
		$\alpha(i)=\lambda_{s+1}-(a_{1,s+1}+\dots+a_{i,s+1})$  for $i=1,\dots,s-1$.
	\end{center}
	\begin{lemma}\label{lemma4}
		The map $g:C \to C'$ is a bijection, where \[(t_1,\dots,t_{s-1},a_{1s}, a_{1,s+1},\dots,a_{s-1,s}, a_{s-1,s+1}) \mapsto (t_1,\dots,t_{s-1}).\]
	\end{lemma}
\begin{proof}
The first and third inequalities of (S$''$4) follow immediately from (S$'$1a) and (S$''$3) respectively. By summing equations (S$''$3)  we obtain $u=\sum_{i=1}^{s-1}a_{is}+\tau_{s-1}$ and using the first inequality of (S$''$2b) we obtain $u-\tau_{s-1}= \sum_{i=1}^{s-1}a_{is} \le \lambda_s$. Hence (S$''$4) is satisfied which means that $\im g \subseteq C'$.

 Suppose $x,x' \in C$, where $x=(t_1,\dots,t_{s-1},a_{1s}, a_{1,s+1},\dots,a_{s-1,s}, a_{s-1,s+1})$ and $x'=(t_1,\dots,t_{s-1},a'_{1s},$ $ a'_{1,s+1},$ $\dots,a'_{s-1,s}, a'_{s-1,s+1})$. From (S$''$3) we have $a_{is}=u_i-t_i = a'_{is}$ for all $i=1,\dots,s-1$. For each such $i$, using (S$''$2a) and what we just showed, \[a_{i,s+1}=\mu_i -\sum_{j=i, \  j \neq s, s+1}^{n}a_{ij}-a_{is} = \mu_i -\sum_{j=i, \  j \neq s, s+1}^{n}a_{ij}-a'_{is} =a'_{i,s+1}.\]
Thus $x=x'$ and $g$ is injective.

Surjectivity is a bit more demanding. Let $y=(t_1,\dots,t_{s-1}) \in C'$ and define for every $i=1,\dots,s-1,$
\begin{equation}\label{defa}
	a_{is}=u_i-t_i, \  
	a_{i,s+1}=\mu_i-\sum_{j=i, \  j \neq s, s+1}^{n}a_{ij}-a_{is}.
\end{equation}
We intend to show that $(t_1,\dots,t_{s-1},a_{1s},a_{1,s+1},\dots,a_{s-1,s}, a_{s-1,s+1}) \in C$, that is $a_{ij} \ge 0$ $ (j=s,s+1)$ and (S$'$1a), (S$''$2), (S$''$3) hold. 

It is clear from the definitions that (S$''$2a) and (S$''$3) hold.

From (\ref{defa}) and the last inequality of (S$''$4), we have $a_{is} \ge 0$. Moreover,  $\sum_{i=1}^{s-1}a_{is}=\sum_{i=1}^{s-1}(u_i-t_i)=u-\tau_{s-1} \le \lambda_{s}$, where the last inequality is due to (S$''$4). Hence the first inequality of (S$''$2b) holds. 

From the hypothesis on $u_1,\dots,u_{s-1}$, there exist \[t'_1,\dots,t'_{s-1},a'_{1s}, a'_{1,s+1},\dots,a'_{s-1,s}, a'_{s-1,s+1} \in \mathbb{N}\] such that \begin{align}\label{1}
	&t-\tau'_{i-1}-\alpha'(i) \le t'_i \le \min\{a'_{i,s+1}, t-\tau'_{i-1}\}, \  i=1,\dots,s-1, \\\label{2}
	&\sum_{j=i, \  j \neq s, s+1}^{n}a_{ij}+a'_{is}+a'_{i,s+1}=\mu_i, \  i=1,\dots,s-1,\\\label{3}
	&\sum_{i=1}^{s-1} a'_{is} \le \lambda_s, \  \sum_{i=1}^{s-1} a'_{i,s+1} \le \lambda_{s+1}, \  \\\label{4}
	&u_i=a'_{is}+t'_i, \  i \le s-1. \\\label{44}
	&t \le \mu_s - \lambda_s +\sum_{i=1}^{s-1}a'_{is}+\tau_{s-1}=\mu_s-\lambda_s-u.
\end{align}
We note the following equalities,
\begin{align}\label{5}
	&a'_{is}+a'_{i,s+1}=a_{is}+a_{i,s+1}, \  i\le s-1, \\\label{6}
	&a'_{is}+t'_{i}=a_{is}+t_{i}, \  i\le s-1,\\\label{7}
	&a'_{i,s+1}-t'_{i}=a_{i,s+1}-t_{i}, \  i\le s-1,\\\label{8}
	&\sum_{i=1}^{s-1}a'_{i,s+1}-\tau'_{s-1}=\sum_{i=1}^{s-1}a_{i,s+1}-\tau_{s-1}.
\end{align}
Indeed, (\ref{5}) follows from (\ref{2}) and the second equality in (\ref{defa}). The equality (\ref{6}) follows from (\ref{4}) and the first equality in (\ref{defa}). From (\ref{5}) and (\ref{6}) we have (\ref{7}) and by summing (\ref{7}) for $i\le s-1$ we obtain (\ref{8}).

From (\ref{7}) we have $a_{i,s+1}= (a'_{i,s+1} - t'_i)+t_i$ and thus 
$a_{is+1} \ge 0 $ for all $i \le s-1$ by the second inequality of (\ref{1}).

From the first inequality of (\ref{1}) and the definition of $\alpha'{(s-1)}$, we have $t-\tau'_{s-2}-\lambda_{s+1}+\sum_{i=1}^{s-1}a'_{i,s+1} \le t'_{s-1}$ and thus  $\sum_{i=1}^{s-1}a'_{i,s+1} - \tau'_{s-1} \le \lambda_{s+1} -t$.
Hence (\ref{8}) implies that	$\sum_{i=1}^{s-1}a_{i,s+1} - \tau_{s-1} \le \lambda_{s+1} -t$. Thus $\sum_{i=1}^{s-1}a_{i,s+1} \le \la_{s+1} -(t-\tau_{s-1}) \le \la_{s+1}$, where the last inequality comes from $t-\tau_{s-1} \ge 0$ of (S$''$4). We have shown the second inequality in (S$''$2b). It remains to be shown that (S$'$1a) holds.
	
	From the first inequality of (\ref{1}) and the definition of $\alpha'{(i)}$, we have $t-\tau'_{i-1}-\lambda_{s+1}+\sum_{u=1}^{i}a'_{i,s+1} \le t'_{i}$, where $i \le s-1$. From this and (\ref{5}) we obtain \[t-\lambda_{s+1} \le \sum_{u=1}^ia'_{us}-\sum_{u=1}^i(a'_{u,s+1}+a'_{us})+\tau'_i=\sum_{u=1}^ia'_{us}-\sum_{u=1}^i(a_{u,s+1}+a_{us})+\tau'_i.\]
 From the above equality and (\ref{6}) we have 
	\[t-\lambda_{s+1} +\sum_{u=1}^{i}a_{u,s+1} \le \sum_{u=1}^ia'_{us}-\sum_{u=1}^{i}a_{us}+
	\tau'_i=\tau_i\]
	and thus $t-\alpha(i)-\tau_{i-1} \le t_i$. We have shown the left inequality of (S$'$1a).
	
	From (\ref{1}) and (\ref{5}) we have \[a'_{is}+t'_i \le a'_{is}+a'_{i,s+1} = a_{is}+a_{i,s+1}.\]
	From (\ref{6}) and the above we have $t_i = a'_{is}+t'_i-a_{is} \le a_{is+1}$ which is one of the right inequalities of (S$'$1a). To show the other, note that from (S$''$4),
	$\tau_i \le \tau_{s-1} \le t $ for $i \le s-1 $, so $t_i \le t - \tau_{i-1} $, for all $i=1, \dots, s-1.$ 
\end{proof}
	\section{Acknowledgments}
	We are very grateful to the anonymous referee for many detailed comments and suggestions that helped to greatly improve the paper. We thank H. H. Andersen for pointing out an error in the Introduction of an earlier draft of the paper. The first author acknowledges the support of the Department of Mathematics, University of Athens.

\end{document}